\documentclass[12pt]{article}
\usepackage[leqno]{amsmath}
\usepackage{amssymb}
\usepackage{amsfonts} 
\usepackage{amstext} 
\usepackage{amsthm} 
\usepackage{euscript} 
\usepackage{eufrak} 
\usepackage{array} 
\usepackage{multirow} 
\usepackage{graphicx}

\makeatletter
\renewcommand{\theenumi}{\roman{enumi}}

\renewcommand{\p@enumii}{\theenumi--}
\makeatother

\newcommand{\Z}{\ensuremath{\mathbb{Z}}}
\newcommand{\N}{\ensuremath{\mathbb{N}}}

\newcommand{\R}{\ensuremath{\mathbb{R}}} 
\newcommand{\C}{\ensuremath{\mathbb{C}}}

\newcommand{\cS}{\ensuremath{\mathcal{S}}}

\newcommand{\cW}{\ensuremath{\mathcal{W}}}

\newcommand{\eG}{\ensuremath{\textsf{G}}} 
\newcommand{\eH}{\ensuremath{\textsf{H}}}

\newcommand{\eS}{\ensuremath{\textsf{S}}} 
 
\newcommand{\eX}{\ensuremath{\textsf{X}}}

\newcommand{\eh}{\ensuremath{\textsf{h}}}

 \newcommand{\fe}{\ensuremath{\EuFrak{e}}}

\newcommand{\fr}{\ensuremath{\EuFrak{r}}}

\newcommand{\pa}{\ensuremath{\partial}}

\newcommand{\h}{\ensuremath{\theta}}

\newcommand{\Symp}{\ensuremath{\text{Symp}}} 
\newcommand{\Ham}{\ensuremath{\text{Ham}}}

\begin{document} 

\begin{center} 

{\Large \bf  Special 
Hamiltonian $S^1$-actions on symplectic 4-manifolds} \\ 

\vspace{.3in} 
Mei-Lin Yau 
\footnote{Research partially supported by MOST grant 109-2635-M-008-002. 


} 
\end{center} 

\vspace{.2in} 
\begin{abstract} 

In this paper we consider symplectic 4-manifolds $(M,\omega)$ with 
$c_1(M,\omega)=0$ which admit 
a Hamiltonian $S^1$-action together with a 
Maslov condition on  orbits of the group action. 
We call such spaces {\em special Hamiltonian $S^1$-spaces} (SHam1-spaces) and 
denote them as $(M,\omega, \Phi, h)$ or $(M,\omega, \Phi)$, where 
$\Phi:S^1=\R/2\pi \Z\to Ham(M,\omega)$, $\Phi(t)=\phi_t$, $\phi_0=id_M$, is a group homomorphism, $h$ is the associated moment map. It turns out 
that there are no compact SHam1-spaces. To minimize topological 
complexity 
we assume that $(M,\omega=d\alpha)$ is an exact, connected open tame manifold with connected level sets $h^{-1}(c)$, $c\in \R$,  the gradient vector field $\nabla h$ of the moment map $h$ is 
complete with respect to some $\omega$-compatible Riemannian metric on $M$, and all orbit spaces $h^{-1}(c)/\Phi$ are homeomorphic to a complete manifold. 
We also assume that the SHam1-action on $(M,\omega=d\alpha)$ is effective, {\em semi-free}, and the fixed point set is finite. SHam1-spaces 
with the above conditions are called {\em simple}. 
 We classify all simple SHam1-spaces and 
show that all of these spaces admit the structure of a  Stein surface. 
Moreover, we show that, a SHam1-action $\Phi$ on a 1-connected Stein surface $\cW_n$ is linear near the $A_n$-string of Lagrangian 
spheres, up to a conjugation with a $\Phi$-equivariant diffeomorphism near the $A_n$-string of Lagrangian 
spheres. 
If $n=0,1$ the then linearity of a SHam1-action $\Phi$ can be 
further expanded to larger domains by employing suitable Liouville functions. 

\end{abstract} 
\newtheorem{prop}{Proposition}[section]
\newtheorem{theo}[prop]{Theorem}
\newtheorem{cond}[prop]{Condition}
\newtheorem{defn}[prop]{Definition}
\newtheorem{exam}[prop]{Example}
\newtheorem{lem}[prop]{Lemma}
\newtheorem{cor}[prop]{Corollary}
\newtheorem{ques}[prop]{Question}

\newtheorem{rem}[prop]{Remark}
\newtheorem{notn}[prop]{Notation}
\newtheorem{fact}[prop]{Fact}


\section{Introduction and main results}

Hamiltonian $S^1$-actions on compact symplectic 4-manifolds have been classified by 
Karshon in \cite{Ka} where she proved that all these Hamiltonian $S^1$-spaces are K\"{a}hler. 
In this article we consider a type of Hamiltonian $S^1$-action  $\Phi:S^1\to Ham(M,\omega)$ on a 
symplectic 4-manifold $(M,\omega)$ 
with an extra property called the {\em Maslov condition},  which 
requests that along any nondegenerate orbit $C$, any $\Phi$-invariant Lagrangian subbundle in the 
symplectic normal bundle 
of $C$, treated as a loop of Lagrangian planes  along $C$, has zero Maslov number. 
We call such a group action a {\em special Hamiltonian $S^1$-action} (SHam1-action) and  the triple $(M,\omega, \Phi)$ a 
{\em special Hamiltonian $S^1$-space} (SHam1-space). 
See Definition \ref{sHam} in \S 2 for the precise definition of SHam1-action. To ensure the 
Maslov condition is permitted we assume that $(M,\omega)$ is  
connected  with the first Chern class $c_1(M,\omega)=0$, and that 
the SHam1-action is  effective on $M$. \\

It turns out that the Maslov condition imposes restrictions on the topology of $M$  as well of the $S^1$-action. 
\begin{theo} \label{specialtop} 
Let  $(M,\omega, \Phi)$ be a connected simple SHam1-space 
and $\Phi$ acts effectively on $M$ with moment map $h$, then 
\begin{enumerate} 
\item $M$ is not compact (Proposition \ref{nocpt}), and 
\item the fixed point set $Fix(\Phi)$  is a discrete set of points (Proposition \ref{fixpt}). 
\item Moreover, if   $(M,\omega=d\lambda)$ is exact then 
 $Fix(\Phi)\subset h^{-1}(0)$ up to addition of a constant to $h$ 
 (Proposition \ref{exact}). 
\end{enumerate} 

\end{theo}

Since a SHam1-space $(M,\omega,\Phi)$ is not compact, if we remove 
a codimension $\geq 1$ 
$\Phi$-invariant subset, e.g., any number of nonconstant $\Phi$-orbits, from $M$, then $\Phi$ restricts to a SHam1-action on the  remaining symplectic manifold $(M', \omega)$, so $(M',\omega,\Phi)$  by itself is also a SHam1-space, but the topology of $M'$ can be arbitrarily complicated. To avoid extra topological complexity of $(M,\omega, \Phi)$ beyond 
the presence of a SHam1-action, we  assume that $(M,\omega=d\alpha)$ is {\em simple}, which is defined as follows: 

\begin{defn} 
{\em Let $(M,\omega=d\lambda, \Phi)$ be a connected open exact SHam1-space, and let $h$ denote the moment map associated to  $\Phi$. $(M,\omega=d\lambda, \Phi)$ is {\em simple} if 
\begin{enumerate} 
\item $M$ is a complete manifold with respect to some Riemannian metric, and is homotopic to a finite handlebody, or equivalently, a finite CW-complex, 
\item the gradient vector field $\nabla h$ of the moment map $h$ is 
complete with respect to some $\omega$-compatible $\Phi$-invariant Riemannian metric on $M$, 
\item level sets $h^{-1}(c)$ are connected for all $c\in \R$, and all 
reduced spaces $h^{-1}(c)/\Phi$ are homeomorphic to a noncompact surface without punctures.   
\end{enumerate} 
}
\end{defn}

In this paper we consider the classification problem of all (connected) simple exact SHam1-spaces $(M,\omega=d\lambda, \Phi, h)$ with finite fixed point set $Fix(\Phi)$. {\em All SHam1-sapces considered 
in this manuscript are assumed to be  simple}. \\ 

For the 1-connected case we  
 obtain the following topological classification results.  
 
 \begin{theo} \label{M=An} 
 Let $(M,\omega, \Phi)$ be a 
connected, 1-connected simple exact SHam1-space, and 
the $\Phi$-action is semi-free  with $n+1$ fixed points for some integer  $n\geq 0$. If $n>0$ then $(M,\omega)$  is symplectically the linear plumbing of $n$ cotangent bundles  of spheres  $S_i\subset M$, $i=1,...,n$, 
 where the union $\cup_{i=1}^nS_i$ is an $A_n$-string of 
 $\Phi$-invariant Lagrangian  spheres. 
  If $n=0$  then $(M,\omega)$ is diffeomorphic to the standard 
  symplectic 4-space $(\R^4, \omega_0=\sum_{i=1^2}(dx_i\wedge dy_i)$.

\end{theo}  

For $n\in \N\cup\{ 0\}$ let $\cW_n\subset\C^3$ denote the Stein surface defined by the equation 
\[ 
z_1^2+z_2^2+z_3^{n+1}=1.  
\] 
Let $\omega_n$ denote the standard symplectic structure on $\C^3$ restricted to $\cW_n$. It is known that for $n>0$ $\cW_n$ is 
topologically the plumbing of  $n$ copies  of cotangent bundles $T^*S^2$ (of Lagrangian spheres) of type $A_n$ \cite{Wu}, and $\cW_0$ is symplectomorphic to 
the standard symplectic 4-space.

 \begin{prop} \label{sWn} 
  Let $\tilde{\Theta}$ denote  the $S^1$-action on $\C^3$ defined by 
 \[ 
 (z_1,z_2,z_3)\to (z_1\cos\h-z_2\sin\h, z_1\sin\h+z_2\cos\h, z_3), \quad \h\in\R/2\pi\Z. 
 \] 
 The action of $\tilde{\Theta}$ preserves $\cW_n$. We denote by $\Theta$ the induced action on $\cW_n$. $\Theta$  acts on 
 $(\cW_n,\omega_n)$ as a SHam1-action with $n+1$ Fixed points, and $\eh:=x_2y_1-x_1y_2$ is the moment map with $Fix(\Theta)\subset \eh^{-1}(0)$.
  \end{prop}

 \begin{rem} 
Note that $(\cW_n,\omega_n=d\alpha,\Theta^{-1})$ is also a SHam1-space, where $\Theta^{-1}$ is the inverse action of that of $\Theta$. 
More generally, if $(M,\omega, \Phi)$ is a SHam1-space, 
 then so is $(M,\omega, \Phi^{-1})$.  We will show in Proposition 
 \ref{exact} that for any connected, exact, simple effective SHam1-space $(M,\omega, \Phi)$, the weight of the $\Phi$-action 
 at its fixed points are either all equal to 1, or all equal to -1, relative 
 to the Darboux charts at each of the fixed points. 
\end{rem}
 
\begin{defn} \label{equivsymp} 
 Two SHam1-spaces $(M_i,\omega_i, \Phi_i)$, i$=1,2$ are 
symplectically  equivariant if there exists a diffeomorphism $f:M_1\to M_2$ such that $f^*\omega_2=\omega_1$ and $f^*\Phi_2=\Phi_1$. 
$f$ is called an {\em isomorphism} between the two SHam-1 spaces 
as in \cite{Ka}. 
\end{defn} 

\begin{prop} \label{1-equiv} Let $\Theta^{-1}$ denote the inverse 
SHam1-action of $\Theta$ on $(\cW_n,\omega_n)$. Then there exists a Hamiltonian diffeomorphism $f$ of $(\cW_n,\omega_n)$, 
such that $f^{-1}\circ\Theta^{-1}\circ f=\Theta$, i.e., actions of $\Theta$ and $\Theta^{-1}$ on $(\cW_n,\omega_n)$ are equivariant 
up to a conjugation by a Hamiltonian diffeomorphism on $(\cW_n,\omega_n)$. 
In particular, we can take 
\[ 
f(x_1,y_1,x_2,y_2,x_3,y_3):=f(-x_1,-y_1,x_2,y_2,x_3,y_3), 
\] 
then $f=f^{-1}$ and $f$ preserves the Lefschetz fibers of $\Theta$ (hence of $\Theta^{-1})$. See Remark \ref{m_f} and Remark \ref{m=-1} for detail. 
\end{prop}

\begin{theo} \label{M=Wn}  
A connected, 1-connected, simple SHam1-space  $(M,\omega=d\lambda, \Phi)$ with $n+1$ fixed points is  symplectomorphic to the  Stein surface $(\cW_n, \omega_n)$. In  particular, the symplectic 
 topology of $(M,\omega)$ is determined by the cardinality of the 
 fixed point set $Fix(\Phi)$. 
 \end{theo}

  Moreover, together with the following theorem by Wu \cite{Wu}, 
 any $A_n$-string of Lagrangian spheres in $(M,\omega)$  
 is Hamiltonian isotopic to any given 
 $A_n$-string of Lagrangian spheres associated to the $\Theta$-action as in Theorem \ref{M=Wn}, up to  a composition of  Lagrangian Dehn twists along the latter:

 \begin{theo}[\cite{Wu}] \label{SympW} 
Any compactly supported symplectomorphism of $\cW_n$ is 
Hamiltonian isotopic to a composition of Dehn twists along 
the standard spheres. In particular, 
$\pi_0(\Symp_c(\cW_n))=Br_{n+1}$.  As a result, 
exact Lagrangians in $A_n$-surface singularities are isotopic to the zero section of a plumbed copy of $T^*S^2$, up 
to a composition of  Lagrangian Dehn twists along the standard 
spheres. 
\end{theo}

We have the following result concerning 
the uniqueness of special Hamiltonian $S^1$-actions on $\cW_n$.

\begin{theo}(Linearity and symplectic equivariance of SHam1-actions) \label{symp-equiv} 
Let  $\Phi$ be any  
semi-free SHam1-action on $(\cW_n,\omega_n)$ 
with moment map $h$ with $n+1$ fixed points, $n\in\Z$, $n\geq 0$, 
and let $S=S_{\phi}=\cup_{i=1}^nS_i\subset \cW_n$ be the associated $A_n$-string of Lagrangian spheres. Also let $\eS=\eS_\Theta$ denote the $A_n$-string of Lagrangian spheres in $W_n$ associated to $\Theta$.

\begin{enumerate} 
\item 
There is a $\Phi$-invariant open neighborhood $U\subset \cW_n$ of $\eS$ (resp. 
the fixed point $(0,0,1)\in \cW_0$ if $n=0$) and a $\Phi$-equivariant  $C^\infty$-diffeomorphism 
$\fe_\eS:U\to V=\fe_\eS(U)\subset\cW_n$ on which $\Phi=\fe_\eS\circ \Theta\circ \fe_\eS^{-1}$, $\omega':=\fe_\eS^*\omega_n$ is symplectic and $\Phi$-invariant, the $\Phi$-action on $U$  is 
linear with respect to the symplectic form $\omega'$. 
Then by Definition  \ref{equivsymp} $(U,\omega', \Phi, h)$ and $(V, \omega_n, \Theta,\eh)$ are 
symplectically equivariant via the symplectomorphism 
$\fe_\eS:(U,\omega')\to (V,\omega_n)$.

\item If $n=0$, any semi-free SHam1-action $\Phi$ on $(\C^2,\omega)$ with $Fix_\Phi=\{ 0\}$ is 
linear up to conjugation and symplectically equivariant to the standard linear action of $\Theta$ on open 4-ball $B_r(0)=\{ q\in \C^2\mid |q|^2<r\}$ centered at $0$ with any prescribed radius $\sqrt{r}$.

\item If $n=1$,  then 
$\Phi$ is symplectically equivariant to $\Theta$ on any 
compactly supported $\Phi$-invariant open domain containing the $A_1$-sphere $S\cong S^2$  associated to the $\Theta$-action on $\cW_1$. 

\end{enumerate} 

\end{theo}

For the case that $(M,\omega=d\lambda, \Phi, h)$ is not 1-connected we assume that 
$Fix(\Phi)$ is not empty (otherwise the $\Phi$-action is free and the 
Maslov condition is vacuous). Then the orbit space $h^{-1}(0)/\Phi$ 
is topologically an oriented noncontractible surface 
of genus $g\geq 0$ with $b\geq 1$ 
boundary components, $g+b\geq 2$, and 
$k:=|Fix(\Phi)|=n+1\geq 1$ marked points.

\begin{theo}  \label{pi1} 
Let $(M,\omega, \Phi)$ be a  connected 
exact SHam1-space with 
$c_1(M,\omega)=0$  and $n+1$ fixed points, $n\geq 0$. 
Assume also that $M$ is not 1-connected. 
Then $M$ admits the structure of a Stein surface, 
which is topologically 
obtained by attaching $(2g+b-1)$-pairs of  1-  and Lagrangian 2-handlebodies  to the boundary of a 
Stein domain $W_n\subset \cW_n$ diffeomorphic to $\cW_n$. The homology groups of $M$ are 
\[
 H_m(M,\Z)=, 
 \begin{cases}  0 & \text{ for } m \neq 0,1, 2, \\ 
 \Z^{n+2g+b-1}& \text{ for } m= 2, \\ 
\Z^{2g+b-1}  & \text{ for } m=1, \\
\Z & \text{ for } m=0. 
\end{cases} 
\] 
In particular, $H_2(M,\Z)$ is generated by $n$ Lagrangian spheres and 
$2g+b-1$ monotone Lagrangian tori. 
\end{theo}

\begin{rem}[Relation with complexity one Hamiltonian space] 
{\rm 
The 4-dimensional exact SHam1-spaces $(M,\omega=d\lambda, \Phi)$ that we consider here are also {\em centered}  complexity one 
spaces as defined in \cite{KaTo1} (see also \cite{KaTo2, KaTo3, KaTo4}) 
except that the moment maps of these spaces are {\em not proper}  unless 
restricted to a $\Phi$-invariant domain $D$ of $M$, so that for any compact 
interval $[a,b]\subset \R$ the preimage 
$h^{-1}([a,b])$ of the moment map is compact. 
} 
\end{rem}

The rest of this paper is organized as follows: In \S\ref{special} we give 
the definition of a SHam1-action on a symplectic 
4-manifold $(M,\omega)$. As an 
example we show that an $S^1$-subgroup of $SU(2)$ acts on  
the standard symplectic 4-space  $(\R^4, \omega_0)$  special Hamiltonianly. Proof of Theorem \ref{specialtop} is given 
in \S\ref{symptop} where we also 
establish various symplectic topological properties of 
SHam1-spaces. 
Theorem \ref{M=An} is proved in \S\ref{1-conn}.  Proposition \ref{sWn} and Theorem \ref{M=Wn} are  verified in \S\ref{Wn} where the $\Theta$-action 
is analyzed and an $A_n$-string of Lagrangian spheres is given. 
This is followed by the proof of Theorem \ref{symp-equiv}.  
Theorem \ref{pi1} is proved in \S\ref{nsimply}. \\


\section{Special Hamiltonian $S^1$-action} \label{special}

\noindent
{\bf Maslov index of a loop of Lagrangian planes}  (see e.g. Ch.2 of \cite{MS}).   \   
Let $L$ be a Lagrangian plane field defined along an oriented  closed curve  
$C$ in a symplectic 4-manifold $(M,\omega)$. Assume that there exists a smooth 
map $\psi:D\to M$ from the unit 2-disc $D$ into $M$ such that $C=\psi(\pa D)$. 
Fix a symplectic trivialization  $\Phi$ of the pullback bundle $\psi^*TM$ over $D$, 
with which the Lagrangian bundle $\psi^*L$ over $\pa D$ can be identified with a 
loop of Lagrangian planes in the symplectic space $\R^4\cong \C^2$ or equivalently, 
the image of a map 
\[ 
\tau:S^1\to \Lambda(\C^2), 
\] 
where $\Lambda(\C^2)$ is 
the space of unoriented Lagrangian planes in $\C^2$. \\ 

The unitary group $U(2)$ acts transitively on  $\Lambda(\C^2)$ with isotropy group $O(2)$, i.e., $\Lambda(\C^2)$ is conjugate to $U(2)/O(2)$. By picking  a unitary basis field  $L$ can be represented by a loop 
(or half of a loop if $L$ not orientable) 
of unitary matrices $A$ in $U(2)$. The Maslov index 
$\mu(\tau)=\mu(L,\psi)$ is then defined to be the degree of the composed map 
\begin{equation} \label{maslov} 
{\det}^2 \circ \tau:S^1\overset{\tau}{\to} U(2)\overset{\det^2}{\to} U(1)\cong S^1. 
\end{equation} 
Several remarks on $\mu(L,\psi)$ are in order: 
\begin{enumerate} 
\item If $M$ is 1-connected then $\mu(L,\psi)$ is always defined since $C$ is the boundary of some disc. 
\item $\mu(L,\psi)$ is independent of the choice of $\psi$. 
\item If $c_1(M,\omega)=0$ Then $\mu(L,\psi)$ is independent of 
the class $[\psi]\in \pi_2(M,C)$  and hence the choice of $\psi$. 
\end{enumerate}

\begin{defn}[{\bf Special Hamiltonian $S^1$-action}] \label{sHam} 
{\rm 
Let $\Phi$ be an $S^1$-group of Hamiltonian diffeomorphisms acting on 
$M$ effectively with nonempty fixed point  set $Fix(\Phi)$. Let $h\in C^\infty(M,\R)$ denote  the 
moment map of $\Phi$, and $X=X_h$ the Hamiltonian vector field on $M$ which 
generates the $\Phi$-action  so that 
\[ 
\omega(X,\cdot)=-dh. 
\] 
$h:M\to \R$ is also called a Hamiltonian function of $(M,\omega)$.  
For $p\not\in Fix(\Phi)$ let $C:=Orb_\Phi(p)$ denote the $\Phi$-orbit of $p$. Take a nonzero $v\in T_pM/T_pC$ so that the plane $X|_p\wedge v$ spanned by  $X|_p$ and 
$v$   is Lagrangian. Extend $X|_p\wedge v$ to an 
oriented  $\Phi$-invariant Lagrangian plane field $L_p$ along $C$ via the $G$-action. The Maslov index $\mu(L_p)$ is independent of the 
choice of nonzero $v\in (T_pM/T_pC)$  with $\omega_p(X,v)=0$. 
Also $\mu(L_p)=\mu(L_q)$ for $q\in Orb_\Phi(p)$. 
We say 
that the $\Phi$ is {\em special} if 
\[ 
\mu(L_p)=0, \quad p\not\in Fix(\Phi) .
\] 
}
\end{defn}

\begin{exam} \label{eG}
{\rm 
Let $M=\R^4=\C^2$ with the standard symplectic structure 
$\omega_0=\sum_{j=1}^2dx_j\wedge dy_j=d\lambda_0$, $\lambda_0=\frac{1}{2}(\sum_{j=1}^2x_jdy_j-y_jdx_j)$ and the standard Euclidean inner product. 
We claim that 
\[ 
\Theta:=\Big\{ \Theta_\h=\begin{pmatrix} \cos \h & -\sin \h\\ \sin \h & \cos\h\end{pmatrix}\Big | \h\in \R/2\pi\Z\Big\}\subset SU(2)  
\]
is a semi-free special Hamiltonian 
$S^1$-group acting on $\C^2$ with 
 $Fix(\Theta)=\{ 0\}$}. \\

 It is easy to see that this $\Theta$-action is semi-free with $Fix(\Theta)=\{ 0\in \C^2\}$, 
 and is Hamiltonian with moment map 
\[ 
\eh:\C^2\to \R, \quad \eh(z_1,z_2)=\Im(z_1\bar{z}_2)=x_2y_1-x_1y_2, 
\] 
the corresponding Hamiltonian vector field is 
\[ 
X=-x_2\pa_{x_1}+x_1\pa_{x_2}-y_2\pa_{y_1}+y_1\pa_{y_2}. 
\]
{\bf Maslov condition.} \ 
We claim that $\Theta$ is special: 
Since the set of nontrivial orbits of $\Theta$ is connected and the Maslov index is discrete, it suffices to show that 
the Maslov condition holds for a typical orbit of $\Theta$, say the loop $C:=Orb_\Theta(p)$ with 
$p=(1,0)\in \C^2$,  
\[ 
C(t)=(\cos t, \sin t)\subset \C^2, t\in \R/2\pi\Z. 
\]  
Then $\dot{C}(0)=X_\eh |_p=\pa_{x_2}$. 
Take $v=-\pa_{x_1}|_p\in T_p\C^2$. Apply $\Theta$ to $X_\eh |_p\wedge v=\pa_{x_2}
\wedge(- \pa_{x_1})|_p$. The resulting Lagrangian plane field along $C$ is 
$L:=\pa_{x_1}\wedge \pa_{x_2}$. The pair $(\pa_{x_1},\pa_{x_2})$ is a unitary 
basis field of the tangent bundle $T\C^2$, also a unitary basis field of $L$.  Therefore 
under the map $\tau$ as in (\ref{maslov}) $L$ can be represented by the constant loop 
$\begin{bmatrix} 1 & 0\\ 0 & 1\end{bmatrix}\subset U(2)$ with constant determinant $1$, hence the Maslov index of $L$ is $\mu(L)=0$. This confirms that the $\Theta$-action on $\C^2$ is special. 
\hfill{$\Box$}
\\ 

\noindent
{\bf Morse index of $\eh$ at $0$.} \ 
Note that $d\eh$ vanishes at $\{ 0\}=Fix(\Theta)$. The Hessian matrix of $\eh$ at the point $0$, with respect to the basis $( \pa_{x_1},\pa_{x_2},\pa_{y_1},\pa_{y_2})$, is 
\[ 
\text{Hess}_0(\eh) = \begin{bmatrix} 0 & 0 & 0 & -1\\ 0 & 0 & 1 & 0 \\ 0 & 1 & 0 & 0 \\ 
-1 & 0 & 0 & 0 \end{bmatrix}, 
\] 
which is nondegenerate with eigenvalues $-1,-1,1,1$, so $\eh$ is a Morse function on $\C^2$, and the Morse index of $\eh$ at $0\in \C^2$ is 2. The corresponding (-1)- and 1-eigenspaces are 
\[ 
E_{-1}=Span\{ \pa_{x_1}+\pa_{y_2}, \ \pa_{y_1}-\pa_{x_2}\}, \quad
E_1=Span\{ \pa_{x_1}-\pa_{y_2}, \ \pa_{y_1}+\pa_{x_2}\}  
\] 
respectively. Also the two 2-dimensional vector spaces $E_{-1}$ and $E_1$ can be identified with 
a pair of complex lines (and hence $\omega_0$-symplectic planes)  in $\C^2$ intersecting transversally 
at $0\in\C^2$: 
\[ 
E_{-1}=\{ z_1+iz_2=0\}, \quad E_1=\{ z_1-iz_2=0\}.  
\] 
\\ 

\noindent
{\bf Stable/unstable submanifolds of $\eh$.} \ 
Let  $\nabla\eh$ 
denote  the gradient vector field of $\eh$ relative to the standard 
Euclidean metric on $\R^4\cong \C^2$, which is both $\omega_0$-compatible and $\eG$-invariant. Then 
\[ 
\nabla \eh=-y_2\pa_{x_1}+x_2\pa_{y_1}+y_1\pa_{x_2}-x_1\pa_{y_2}. 
\] 
Let $W_{\eh}^s(0)$ and $W_{\eh}^u(0)$ denote the stable and 
unstable submanifolds of of the flow of $\nabla\eh$ 
 at $0\in\R^4\cong \C^2$ respectively. 
We claim that  
\[ 
W_{\eh}^s(0)=E_{-1} =\{ z_1+iz_2=0\}, \quad 
W_{\eh}^u(0)=E_1 =\{ z_1-iz_2=0\}. 
\]

Consider the point $(1,i)\in \C^2$ on $E_{-1}$, the integral curve of $\nabla \eh$ 
passing through the point $(1,i)$ at time $t=0$ is 
$\ell_{(1,i)}(t)=(e^{-t}, ie^{-t})\subset E_{-1}$, the half $\R$-space spanned by 
$\nabla \eh(1,i)=-(\pa_{x_1}+\pa_{y_2})$. The $\Theta$-action preserves 
the space of integral curves of $\nabla \eh$. In particular 
$\Theta_{\pi/2}(\ell_{(1,i)})=\ell_{(-i,1)}\subset 
E_{-1}$ is the integral curve of $\nabla \eh$ 
passing through the point $(-i,1)$ at time $t=0$, also a half $\R$-space spanned by 
$\nabla \eh(-i,1)=\pa_{y_1}-\pa_{x_2}\subset E_{-1}$. Moreover, 
$Orb_\Theta(\ell_{(1,i)})=E_{-1}\setminus\{ 0\}$. So $W_{\eh}^s(0)=E_{-1}=\{ z_1+iz_2=0\}$. \\ 

The case for $W_{\eh}^u(0)$ can be verified in a similar way. 
Consider the point $(1,-i)\in \C^2$ on $E_1$, the integral curve of $\nabla \eh$ 
passing through $(1,-i)$ at time $t=0$ is 
$\ell_{(1,-i)}(t)=(e^t, -ie^t)\subset E_1$, the half $\R$-space spanned by 
$\nabla \eh(1,-i)=\pa_{x_1}-\pa_{y_2}$. $\Theta_{\pi/2}(\ell_{(1,-i)})=\ell_{(i,1)}\subset 
E_1$ is the integral curve of $\nabla \eh$ 
passing through the point $(i,1)$ at time $t=0$, also a half $\R$-space spanned by 
$\nabla \eh(i,1)=\pa_{y_1}+\pa_{x_2}\subset E_1$. Moreover, 
$Orb_\Theta(\ell_{(1,-i)})=E_1\setminus\{ 0\}$. So $W_{\eh}^{u(0)}=E_1=\{ z_1-iz_2=0\}$. 
\hfill{$\Box$} 
\end{exam} 

\begin{rem} \label{m_f}
{\rm Observe that $\Theta^{-1}$ 
 is the inverse action of $\Theta$ on $\C^2$, 
with the pair $(-X, -\eh)$ as its corresponding Hamiltonian vector field and moment map instead, and orbits of the $\Theta^{-1}$-action are precisely the orbits of $\Theta$-action but with the opposite orientation. Nevertheless, $\Theta$ and $\Theta^{-1}$ act on $\C^2$ 
symplectically equivariant. For example, consider 
the symplectic diffeomorphism $f_0:\C^2\to \C^2$ given by 
\[ 
f_0(x_1,y_1,x_2,y_2)=(-x_1,-y_1,x_2,y_2), 
\] 
$f_0$ is a linear symplectic map which can be represented by the unitary matrix 
\[ 
\begin{bmatrix} -1 & 0 \\ 0 & 1\end{bmatrix}, 
\] and $f_0=f_0^{-1}$. 
A direct computation 
shows that 
\[ 
f_0^{-1}\circ \Theta^{-1}\circ f_0=\begin{bmatrix} -1 & 0 \\ 0 & 1\end{bmatrix} \begin{bmatrix} \cos\h & \sin\h\\ -\sin\h & \cos \h\end{bmatrix} \begin{bmatrix} -1 & 0 \\ 0 & 1\end{bmatrix}
=\begin{bmatrix} \cos\h & -\sin\h\\ \sin\h & \cos \h\end{bmatrix} 
=\Theta.  
\]  
So $\Theta$ and $\Theta^{-1}$ are equivariant SHam-1 actions on 
$\C^2$. 
}
\hfill{$\Box$}


\end{rem}

\section{Some topological constraints imposed by Maslov condition} 
\label{symptop} 

It turns out that at every isolated fixed point $p\in Fix(\Phi)$ of a special Hamiltonian 
$S^1$-space $(M,\omega, \Phi)$, the $\Phi$-action is Hamiltonian isotopic 
to the standard $\Theta$-action in a Darboux chart of $p$.  
\\ 

\noindent
\begin{prop}  \label{G=eG} 
Let $(M,\omega, \Phi)$ be a SHam1-space 
with moment map $h$, and $p\in Fix(\Phi)$ an isolated fixed point of 
the $\Phi$-action. 

\begin{enumerate}

\item At $p$  the $\Phi$-action induces a linear symplectic $S^1$-action 
on the tangent space $T_pM$ which can be identified with $\Theta\subset SU(2)$ 
with respect to a suitable symplectic Darboux chart of $p\in M$. 

\item There is a $\Phi$-invariant chart $(V, x_1,x_2,y_1,y_2)$ centered at $p$ and $\Phi$-equivariant with respect to a linear action 
of $\Phi$ on $\R^4\cong \C^2$ such that $\omega|_V=\sum_{k=1}^2dx_k\wedge dy_k$.

%
%

\item  All isolated critical points of $h$ are nondegenerate with  Morse index $2$. 

\end{enumerate} 
\end{prop} 

\begin{proof} 
(i). 
 Parametrize $\Phi$ by $t\in \R/2\pi\Z$ so that $\Phi_0$ is the identity map. Fix $p\in Fix(\Phi)$. By Darboux-Weinstein Theorem we can symplectically identify a 
small open neighborhood $V\subset M$ of $p$ with a open neighborhood of 
the origin $0\in (\R^4,\omega_0)$ of the standard symplectic $\R^4$  via a local diffeomorphism $\phi:(V,\omega)\to (\phi(V)\subset \R^4, \omega_0)$ with $\phi(p)=0$ and $\phi^*\omega_0=\omega$. 
Then the linear map $A_t:=(\Phi_t)_*|_p:T_pM\to T_pM$ 
is symplectic, i.e., $A_t\in Sp(4,\R)$ for all $t$, 
with $A_0$ being the identity map. Since 
$\Phi_t\circ \Phi_s=\Phi_{t+s}=\Phi_s\circ \Phi_t$ for all $s,t$, we have $A_tA_s=A_sA_t$ 
for all $s,t$ as a result. I.e., $A_t$, $t\in \R/2\pi\Z$, is an $S^1$-subgroup of 
$Sp(4,\R)$. Up to conjugation (or equivalently a choice of $\omega_0$-compatible complex structure) we may assume that $A$ is an $S^1$-subgroup of 
the unitary group $U(2)=Sp(4,\R)\cap O(4,\R)$ with respect to the 
standard complex structure $J_0$ on $\R^4\cong \C^2$. 
Now $A$ is either the 
centralizer $S^1$-group $C$ of $U(2)$  or an $S^1$-subgroup of $SU(2)$.  \\ 
 
 If $A=C$ then $A=\Big\{ A_t=\begin{bmatrix} e^{ it} & 0 \\ 0 & e^{ it} \end{bmatrix}\Big | t\in \R/2\pi \Z\Big\}$. The linear $A$-action is 
generated by the Hamiltonian vector field $X^A:=\sum_{k=1}^2(-y_k\pa_{x_k}+x_k\pa_{y_k})$. To check the Maslov condition, pick a 
point  $q\not\in Fix(A)$ say, $q=(-i,0)\in \C^2$ where $X^A_q=\pa_{x_1}$. Let  $v_q=\pa_{x_2}$. The Lagrangian plane field $L=Orb_A(\pm\pa_{x_1}\wedge \pm\pa_{x_2})$ along 
$\gamma:=Orb_A(q)$ is 
$L(t)=(\cos t\pa_{x_1}+\sin t\pa_{y_1})\wedge (\cos t\pa_{x_2}+\sin t\pa_{y_2})$,  $t\in \R/2\pi\Z$. With its unitary basis field  $(\cos t\pa_{x_1}+\sin t\pa_{y_1}, \cos t\pa_{x_2}+\sin t\pa_{y_2})$ 
$L$ can be represented by the loop of unitary matrices 
$\begin{bmatrix} e^{it} & 0\\ 0 & e^{it}\end{bmatrix}$. By (\ref{maslov}) the Maslov index of $L$ is $\mu(L)=4\neq 0$. So 
$A\neq C$. Similarly $A\neq C^{-1}$ as otherwise we would have 
$\mu(Orb_A)(\pa_{x_1}\wedge \pa_{x_2})=-4\neq 0$. 
So $A$ is an $S^1$-subgroup of  $SU(2)$. 
Since all $S^1$-subgroups of $SU(2)$ are conjugate we have 
$A=\Theta$ or $\Theta^{-1}$ up to a change of the complex coordinates of $\C^2$
via a special unitary linear map given by some element of $SU(2)$.  
 \\ 

As the set $Fix(\Phi)$ is discrete without accumulation points we may 
choose for each $p\in Fix(\Phi)$ a symplectic Darboux chart $U_p$ such that 
$U_p\cap U_q=\emptyset$ for $p,q\in Fix(\Phi)$ with $p\neq q$. The above 
result about $A$ applies to every $p\in Fix(\Phi)$ as a result.  
This finishes the proof of (i).  \\

(ii) It follows from (i) that  an equivariant version of Darboux Theorem \cite{DelMel} (see also \cite{GuillStern} Theorem 22.2) applies to every 
isolated fixed point $p\in Fix(\Phi)$. Namely, for each isolated fixed point 
$p\in Fix(\Phi)$, there exists a $\Phi$-invariant chart $V$ centered at  $p\in M$  and  $\Phi$-equivariant with respect to a linear action of $\Phi$  on $\R^4$ so that  
$\omega|_V=\omega_0$ is the standard symplectic form 
$\omega_0=\sum_{k=1}^2dx_k\wedge dy_k$ on $\R^4$, and $f^{-1}\circ \Phi\circ f=\Theta\subset  SU(2)$. \\

 For (iii), observe that $dh|_p=-\omega(X_h|_p,\cdot )$ for $p\in M$. Since $\omega$ 
is nondegenerate, $\omega(X_h|_p,\cdot )=0$ iff $X_h$ vanishes at $p$, i.e., iff 
$p\in Fix(\Phi)$, i.e.,  $Crit(h)=Fix(\Phi)$. By (ii) we may 
assume that $\Phi=\Theta$ and hence $h=\eh$ near each isolated critical point 
$p\in Fix(\Phi)$ of $h$, therefore all isolated critical  points of $h$ are nondegenerate 
with Morse index $2$. 
\end{proof}

Forgetting the Maslov condition, 4-dimensional SHam1-spaces are  by themselves (not necessarily compact) {\em Hamiltonian $S^1$-spaces}. Recall  that 4-dimensional compact  Hamiltonian $S^1$-spaces
have been classified by Karshon \cite{Ka}: 

\begin{theo}[\cite{Ka}, Theorem 6.3]  
Every compact 4-dimensional Hamiltonian $S^1$-space 
$M$  can be obtained 
by a sequence of $S^1$-equivariant symplectic blow-ups  from one of the 
 following two types of "minimal model spaces" at their fixed points: 
\begin{enumerate} 
\item the complex projective plane $\C P^2$ or a Hirzebruch 
surface, with a symplectic form that comes from a K\"{a}hler form, if $M$ 
has at most one fixed surface (a symplectic sphere); or 
\item a ruled manifold, i.e. a smooth manifold 
which is topologically an $S^2$-bundle over a closed surface, with a compatible $S^1$-action 
that fixes the base surfaces and rotates the fibers (\cite{Ka} Definition 6.13, Lemma 6.15).  A ruled manifold admits a compatible 
K\"{a}hler structure, has two fixed surfaces and no interior 
fixed points .
\end{enumerate} 
\end{theo}  

Recall that if $(M,\omega,\Phi)$ is a  4-dimensional compact Hamiltonian $S^1$-space then each connected component of $Fix(\Phi)$ is either a single point or a symplectic surface, and the maximum and minimum of the moment map $h$ is each attained on exactly one component of $Fix(\Phi)$ (see  e.g. \cite{Ka} Appendix A). \\ 

The following two propositions together verify Theorem \ref{specialtop}.  

\begin{prop} \label{nocpt} 
A  closed compact symplectic 4-manifold $(M,\omega)$  does not admit a 
SHam1-action. 
\end{prop} 

\begin{proof} 
Assume in contrary that $(M,\omega)$  allows a special SHam1-action $\Phi:S^1\to Ham(M,\omega)$. We may assume that the action is effective. Let 
 $h$ denote its moment map.  Then 
 $Fix(\Phi)=Crit(h)$.  $M$ is compact, $h$ is bounded, and the sets 
 $M_-:=\{ p\in M\mid h(p)=\min h\}$ and $M_+:=\{ p\in M\mid h(p)=\max h\}$ are both nonempty. Since every isolated fixed point 
 of the $\Phi$-action has Morse index 2, the set $M_-$ contains no connected component with isolated fixed points, and neither does $M_+$. Then both $M_-$ and $M_+$ are connected compact closed symplectic surfaces and, following \cite{Ka}, $M$ can be obtained from  a {\em ruled manifold}  with two fixed surfaces  and 
 no interior fixed points by a  sequence of equivalent blowups at 
 fixed points that are not minima for the moment map of the ruled manifold. \\

 Suppose that $M$ is a ruled manifold, then 
the orbit $\gamma_q:=Orb_G(q)$ of $q\not\in Fix(G)$ is contained some 
symplectic $S^2$-fiber  
$S$ of $M$. But $c_1(M)([S])=S\cdot S +e(S)=0+2\neq 0$, $c_1(M.\omega)\neq 0$. Indeed, let $F_p$ denote the $S^2$ fiber over  $p\in M_-$. Fix any $q\in F_p\setminus (M_+\cup M_-)$. $Orb_\Phi(q)$ divides 
$F_p\cong S^2$ into two discs $D_-\cup D_+$ with $D_-\cap M_-=\{p\}$. 
We may identity a $\Phi$-invariant symplectic Darboux neighborhood of $p\in M$ 
as $D'\times D_\epsilon\subset \C_{z_1}\times \C_{z_2}$ so that 
$D_-\subset D'\times \{ 0\}$, and $D_\epsilon$ is a coordinate chart of $F_p$ centered at $p$. Up to a choice of orientation of $\gamma_q$ we 
may assume that the tangent vector field along $\gamma_q$ is $X:=
-y_1\pa_{x_1}+x_1\pa_{y_1}$. Let $v:=\pa_{x_2}|_{\gamma_q}$ be 
the vector field $\pa_{x_2}$ restricted to $\gamma_q$. Then 
$L:=X\wedge v$ is a Lagrangian plane field along $\gamma_q$ 
with Maslov index $\mu(L,D_-)=2$. On the other hand, if we choose 
$D_+$ as the disc bounded by $\gamma_q$, then the orientation of 
$\gamma_q$ as the boundary of $D_+$ is given by $-X$, and the 
corresponding Maslov index becomes $\mu(L=(-X)\wedge v, D_+)=-2$. 
Either way the Maslov condition is not satisfied. So $M$ cannot be a ruled manifold, nor 
can it  be obtained from any ruled manifold by a sequence of equivalent blowups at fixed points that are not minima for the moment map of the ruled manifold. 
We conclude that a 4-dimensional compact symplectic 
 manifold $(M,\omega)$  does not admit an effective special Hamiltonian $S^1$-action. 
\end{proof}

\begin{prop}[Fixed point set is discrete]  \label{fixpt} 
 Let  $(M,\omega)$ be a connected noncompact symplectic 4-manifold. Suppose that there is 
 a special Hamiltonian $S^1$-action $\Phi:S^1\to Ham(M,\omega)$ on $M$ 
 with contractible orbits and nonempty fixed point set $Fix(\Phi)$, then $Fix(\Phi)$ consists of isolated points. 
 \end{prop} 
 
 \begin{proof}

 Let $\Phi:S^1\to Ham(M,\omega)$ be a semi-free SHam1-action on $(M,\omega)$. Denote by $X$  the 
 Hamiltonian vector field on $M$ which generates the $\Phi$-action, and 
 $h:M\to \R$ the corresponding moment map defined by 
 $\omega(X,\cdot )=-dh$. By assumption all $\Phi$-orbits are contractable and $Fix(\Phi)$ is not empty. It is known that 
every connected component of $Fix(\Phi)$ is either a single point or a 
symplectic surface (\cite{Ka}, Appendix A). 
If $Fix(\Phi)$ contains  a connected symplectic 
surface $\Sigma$ then $h$ is constant on $\Sigma$  and 
$h(\Sigma)$ is a local extremum of $h$. \\ 

Fix a point $p\in \Sigma$ and an open neighborhood $U_p\subset \Sigma$ of $p$ so that $\Phi$ acts nontrivially near $U_p$ except on $U_p$. 
There is a $\Phi$-invariant tubular neighborhood of $U_p$ 
which is symplectomorphic to a trivial disc bundle $E\cong D\times U_p$ over $U_p$ on which $\Phi$ acts nontrivially as rotations 
along fibers of $E$, fixing $U_p$ pointwise. Since $\Sigma$ 
is symplectic we may identify $E=D\times U_p$ with 
$D_1\times D_2\subset \C\times \C$ with $U_p=\{ 0\} \times D_2$ 
being  the base disc centered at $p=(0,0)$, and $D_1\times \{ pt\}$ as 
fibers, with $\Phi$ acts $D_1\times D_2$ by $t\cdot (z_1,z_2)=(e^{\pm int}z_1,z_2)$, $t\in \Phi\cong \R/2\pi\Z$. \\ 

Let $\gamma$ be a $\Phi$-orbit in the disc fiber over $p$ and $D_\gamma\subset D_1\times \{ 0\}$ the fiber disc over $p$ bounding 
$\gamma$. Assume that $h(\Sigma)$ is a local minimum of $h$ then 
$h=\frac{1}{2}|z_1|^2+constant$ on $E$ and 
the tangent vector field of $\gamma$ is $X=-y_1\pa_{x_1}+x_1\pa_{y_1}$. Let $v:=\pa_{x_2}|_\gamma$.  Then 
$L:=X\wedge v$ is a loop of Lagrangian planes along $\gamma$ 
with maslov index $\mu(L, D_\gamma)=2\neq 0$. On the other hand, 
if $h(\Sigma)$ is a local maximum of $h$ then 
$h=-\frac{1}{2}|z_1|^2+constant$ on $E$ and 
the tangent vector field of $\gamma$ is $X=y_1\pa_{x_1}-x_1\pa_{y_1}$. In this case $\mu(L,D_\gamma)=-2\neq 0$. Either way the Maslov condition is not met. 
So $Fix(\Phi)$ does not 
contain any  symplectic surface.  
We conclude that $Fix(\Phi)$ is a discrete set of points. 
\end{proof}


\begin{prop} \label{exact} 
Let $(M,\omega=d\lambda, \Phi)$ be a connected,  exact, simple SHam1-space,  and let $\lambda$ be a $\Phi$-invariant primitive 1-form of $\omega$. Then 
$Fix(\Phi)\subset h^{-1}(0)$ up to addition of a constant to $h$, 
\end{prop} 

\begin{proof} 
Let $X$ be the Hamiltonian vector field of the moment map $h$ associated to the $\Phi$-action. 
Observe that
\begin{equation} 
-dh=d\lambda(X_h,\cdot)=L_{X_h}\lambda-d(\lambda(X_h))=-d(\lambda(X_h))
\end{equation} 
with $L_{X_h}\lambda=0$
since $\lambda$ is $\Phi$-invariant, so 
$\lambda(X_h)-h$ is  a constant. Up to an addition of a 
constant to $h$ we may assume that 
\begin{equation} \label{Xh0} 
h=\lambda(X_h) \quad \text{and hence \quad $\lambda(X_h)=0$ on $h^{-1}(0)$},  
\end{equation} 
then 
\[ 
Fix(\Phi)
\subset h^{-1}(0)
\] 
 since $\lambda(X_h)$ is constant on level 
sets of $h$.  \end{proof}


\section{Simply-connected Exact Simple SHam1-spaces} \label{1-conn} 

The rest of this paper will focus on the 
classification of  connected exact simple SHam1-spaces. 
We assume that {\em $Fix(\Phi)$ is a nonempty finite set} as before. 
Let 
\[ 
M_c:=\{ h^{-1}(c)\} , \quad \fr_c=\fr|_{h^{-1}(c)}:h^{-1}(c)/\Phi ,  \quad c\in \R. 
\]

\begin{prop} \label{morse} 
Let $(M,\omega=d\lambda, \Phi)$ be a connected, 1-connected 
exact  simple SHam1-space with $c_1(M,\omega)=0$, $\lambda$ being 
$\Phi$-invariant,  
$\Phi$ acts semi-freely on $M$ with moment map $h:M\to \R$, 
and $Fix(\Phi)\subset h^{-1}(0)$ is a finite set. Then the followings are true. 

\begin{enumerate} 

 \item $h^{-1}(c)/\Phi$ is 1-connected hence homeomorphic to a disc. 
\item 
 $\lambda(X_h)=h$ on $M$. Moreover $d\lambda(X_h,\cdot)=0$ 
on $h^{-1}(0)$. In particular, for $q\in h^{-1}(0)\setminus Fix(\Phi)$, the symplectic normal space $T^\omega_q(h^{-1}(0)):=\{ v\in T_qM\mid \omega_q(v, u)=0, \  
\forall u\in T_q(h^{-1}(0))\}$ is generated by $(X_h)_q$.  
\end{enumerate} 
\end{prop}

\noindent
{\em Proof.} {\em (i)} 
Recall that $h^{-1}(c)$ is connected for all  $c$. Observe that 
the gradient flows of $\pm\nabla h$ induce a deformation retract from 
$M$  to $h^{-1}(0)$, keeping $h^{-1}(0)$ fixed all the time. So  $M$ is 
homotopic to $h^{-1}(0)$. Since $M$ is 1-connected, so is $h^{-1}(0)$. 
$\fr_0:h^{-1}(0)\to h^{-1}(0)/\Phi$ is a singlar fibration over $h^{-1}(0)/\Phi$, 
hence $h^{-1}(0)/\Phi$ is 1-connected. Moreover, the gradient flows of $\pm\nabla h$  induce a homeomorphism between the quotient 
spaces $h^{-1}(c)/\Phi$  and 
$h^{-1}(0)/\Phi$ for all  $c\in \R$, so $h^{-1}(c)/\Phi$ is 1-connected 
for all $c$ provided that $M$ is 1-connected. 
\\

\noindent
{\em (ii)}
 $\lambda$ is a $\Phi$-invariant primitive 1-form of $\omega$, hence 
$0=L_{X_h}\lambda=d(\lambda(X_h))+d\lambda(X_h,\cdot)=d(\lambda(X_h))-dh$. 
Since $d(\lambda(X_h))-dh=0$ on $h^{-1}(0)$, $\lambda(X_h)$ is constant on $h^{-1}(0)$. $h^{-1}(0)$ is connected and $\lambda(X_h)=0$ at $p\in Fix(\Phi)$ so 
$\lambda(X_h)=0$ and hence $d\lambda (X_h,\cdot)=0$ on $h^{-1}(0)$.
Therefore $(X_h)_q\in T^\omega_q(h^{-1}(0))$ all  $q\in h^{-1}(0)\setminus Fix(\Phi)$. Since $\dim (T^\omega_q(h^{-1}(0)) +\dim T_q(h^{-1}(0))=\dim 
T_qM=4$, $X_h \subset T(h^{-1}(0))$  generates the field of symplectic normal spaces of $T(h^{-1}(0)\setminus Fix(\Phi))$. 
\hfill{$\Box$}

\begin{rem} 
{\rm 
  Property (ii) above implies that for any embedded smooth curve $\gamma\subset M_0/\Phi$, its 
preimage $L_\gamma:=\fr_0^{-1}(\gamma)\subset h^{-1}(0)$ is an immersed 
monotone Lagrangian surface of $(M,\omega)$. 
For example, if $\gamma$ is a circle missing the  set $(Fix(\Phi)$ then 
$L_\gamma$ is an embedded monotone Lagrangian torus. If  
$\gamma$ is a circle which meets the  set $\fr_0(Fix(\Phi))$ at exactly one point, then $L_\gamma$ is a Lagrangian Whitney sphere, i.e., an immersed 
monotone Lagrangian torus with one transversal self-intersection point. 
If $\gamma$ is an embedded arc which connects  exactly two distinct points 
in $\fr_0(Fix(\Phi))$ and at its endpoint, then $L_\gamma$ is an embedded 
Lagrangian sphere. 
}
\end{rem}

It turns out that a 1-connected exact SHam1-space $(M,\omega, \Phi)$ with finite 
$Fix(\Phi)$  is symplectically the plumbing of a 
finite number of  $T^*S^2$. \\ 

\noindent 
{\bf Proof of Theorem \ref{M=An}.}\

Let $\gamma:(-\infty,\infty)\to h^{-1}(0)/\Phi\cong \C$ be a smoothly 
embedded curve so that the complement $(h^{-1}(0)/\Phi)\setminus \gamma$ consists of two connected components, and 
$\gamma(t_i)=p_i$ for some 
$0=t_0<t_1<t_2<\cdots<t_n<\infty$, where $\{ p_0,p_1,\dots, p_n\}=Fix(\Phi)$. Let 
$\gamma_i:=\gamma([t_{i-1},t_i])$, 
$\gamma_0=\gamma((-\infty,0])$, $\gamma_{n+1}=
\gamma([t_n,\infty))$. 
Let $S_i:=\fr_0^{-1}(\gamma_i)\subset h^{-1}(0)$, then by Proposition \ref{morse} (ii) $S_i$ is a $\Phi$-invariant Lagrangian sphere which is smooth except perhaps at the two points $p_{i-1}$ and $p_i$. \\

Let $U_i\subset M$ be a $\Phi$-invariant symplectic Darboux chart  centered at $p_i$, $U_i\cap U_j=\emptyset$ for  $i\neq j$. Let $U:=\cup_{i=0}^nU_i$. 
By shrinking $U_i$ if necessary 
we may assume that $\Phi=\Theta$ on $U$. 
 Then $h=\eh$ and $X_h=X_\eh$ on $U$. Each $U_i$ is symplectically identified with a open neighborhood 
of $0\in\C^2$ with $p_i$ identified with the origin $0\in\C^2$. 
Recall 
 $\fr_0=\fr|_{h^{-1}(0)}/\Phi$. 
We may assume that $\fr(z_1,z_2)=\frac{1}{2}(z_1^2+z_2^2)+\gamma(t_i)$ when restricted to  each $U_i$, and, by perturbing $\gamma$ if necessary, 
that  $\fr_0^{-1}(\gamma_i)\cap (U_{i-1}\cap h^{-1}(0))=\{y_1=0=y_2\}\cap U_{i-1}$ and $\fr_0^{-1}(\gamma_i)\cap (U_i\cap h^{-1}(0))=\{x_1=0=x_2\}\cap U_i$ are Lagrangian discs for each $i$. 
Then 
\[ 
S_i:=\fr_0^{-1}(\gamma_i), \quad i=1,...,n, 
\] 
 are smoothly 
embedded Lagrangian spheres, 
\[ 
S_i\pitchfork S_{i+1}=\{ p_i\} \ \text{ for } i=1,...,n-1, \quad 
S_i\cap S_j=\emptyset \  \text{ if } |i-j|\neq 1. 
\] 
$S_i$ can be oriented so that the intersection number is $S_i\cdot S_{i+1}=1$ for $i=1,...,n-1$. 
The union $\cS:=\cup_{i=1}^n
S_i$  is an $A_n$-string of Lagrangian spheres. \\

By Lagrangian Neighborhood Theorem \cite{MS,W}, a tubular neighborhood $V_i\subset M$ 
of the Lagrangian sphere $S_i$ is symplectomorphic to 
a tubular neighborhood of the 0-section of the cotangent 
bundle $(T^*S_i,\omega_{can})$ of $S_i$ equipped with the canonical 
symplectic structure. $T^*S_i$  and $T^* S_{i+1}$  are plumbed together near $p_i$ so that 
tangent and cotangent spaces of $S_i$ are identified 
with the cotangent and tangent spaces of $S_{i+1}$ respectively for 
$i=1,...,n-1$: 
\[ 
T_{p_i}S_i=\pa_{x_1}\wedge \pa_{x_2}=T^*_{p_i}S_{i+1}, \quad 
T^*_{p_i}S_i=\pa_{y_1}\wedge \pa_{y_2}=T_{p_i}S_{i+1},  
\]
with $T^*S_i$ and $T^*S_{i+1}$ plumbed near $p_i$ via the identification $T_{p_i}S_i\cong 
T^*_{p_i}S_{i+1}$ with the following correspondence on local coordinates 
at $p_i$: 
\[ 
T^*_{p_i}S_i \ni (x,y) \to (-y,x)\in T^*_{p_i}(S_{i+1}), 
\] 
where $x=(x_1,x_2)$ are coordinates of  $S_i$ near $p_i$, $y=(y_1,y_2)$ are the fiber coordinates of the cotangent space $T^*_{p_i}S_{i}$. 
\\

Indeed we can take an open cover of $\fr_0=h^{-1}(0)/\Phi\cong \R^2$ by 
1-connected domains $V_i\subset \fr_0$, $i=1,\cdots,n$ so that $\fr_0=\cup_i V_i$, 
$\gamma_i\subset V_i$, $V_i\cap V_j=\emptyset$ if $|i-j|>1$, and $V_i\cap V_{i+1}$
is 1-connected. 
Then each $\fr^{-1}(V_i)$ is symplectomorphic to 
an open neighborhood in $T^*S_i$ of the zero section $S_i$ of $T^*S_i$, 
and $V_i\cap V_{i+1}$ is 
where the plumbing of $T^*S_i$ and $T^*S_{i+1}$ takes place. 
This completes the proof of Theorem \ref{M=An} if $|Fix(\Phi)|=n+1\geq 2$. \\ 

If $|Fix(\Phi)|=1$ then $M$ is diffeomorphic to a symplectic 4-space, and we may take 
the fixed point to be $0\in \R^4$. 
\hfill{$\Box$}
\\

\section{A special Hamiltonian $S^1$-action on $\cW_n$} \label{Wn}

We start with a verification of Proposition \ref{sWn}: \\ 

For each  $n\in \N\cup\{ 0\}$ consider the Stein surface in $\C^3$ 
given by 
\[ 
\cW_n: z_1^2+z_2^2+z_3^{n+1}=1, \quad (z_1,z_2,z_3)\in \C^3. 
\]   

The standard SHam1-action $\Theta$ on $\C^2$  can be lifted to a SHam1-action on $\C^3$ which acts trivially on the $z_3$-coordinate: 
\[
\tilde{\Theta}:=\Big\{ \tilde{\Theta}_\h=\begin{pmatrix} \cos \h & -\sin\h & 0 \\ 
\sin\h & \cos\h & 0 \\ 0 & 0 & 1\end{pmatrix}\mid \h\in\R/2\pi\Z\Big\}\subset SU(3).  
\] 
The action of $\tilde{\Theta}$ preserves $\cW_n$ and restricts to a special 
Hamiltonian $S^1$-action on $\cW_n$, denoted as $\Theta$ for simplicity,  with moment map 
\[ 
\eh=\Im(z_1\bar{z}_2)=x_2y_1-x_1y_2 \in C^\infty(\C^3), 
\] 
the corresponding Hamiltonian vector field is 
\[ 
X=-x_2\pa_{x_1}+x_1\pa_{x_2}-y_2\pa_{y_1}+y_1\pa_{y_2}, 
\] 
and the fixed point set is 
\[ 
Fix(\Theta)=\{ p_k:=(0,0,\xi^k) | k=0,...,n\}\subset \eh^{-1}(0), 
\quad \xi:=e^\frac{2\pi i}{n+1}. 
\] 

The tangent space of $\cW_n$ at $p_k$ is $T_{p_k}\cW_n=\C^2_{z_1z_2}$ for each $k$. Since $\eh$ is independent of $z_3$, the Morse index of $\eh$  at $p_k$ is $2$ as in Example \ref{eG}. 
It follows that the $\Theta$-action on $\cW_n$ is special Hamiltonian.  
This completes the proof of Proposition \ref{sWn}. 
\\ 

\begin{rem} \label{m=-1} 
{\em (Equivarience between the $\Theta$-action  and its inverse  $\Theta^{-1}$ on $(\cW_n,\omega_n)$.)} \\

{\rm On $(\cW_n,\omega_n)$ consider the symplectic diffeomorphism 
$f:\cW_n\to \cW_n\subset \C^3$ 
\[ 
f(x_1,y_1,x_2,y_2,x_3,y_3)=(-x_1,-y_1,x_2,y_2,x_3,y_3). 
\] 
$f$ can be represented by the unitary matrix 
\[ 
\begin{bmatrix} -1 & 0 & 0 \\ 0 & 1 & 0\\ 
0 & 0 & 1\end{bmatrix}, \quad f^{-1}=f. 
\]

Then 
\[ 
f^{-1}\circ \Theta^{-1}\circ f=\begin{bmatrix} -1 & 0  & 0 \\ 0 & 1 & 0 \\ 0 & 0 & 1\end{bmatrix} \begin{bmatrix} \cos\h & \sin\h & 0 \\ -\sin\h & \cos \h & 0 \\ 0 & 0 & 1 \end{bmatrix} 
\begin{bmatrix} -1 & 0 & 0  \\ 0 & 1 & 0 \\ 0 & 0 & 1\end{bmatrix}
\] 
\[ 
=\begin{bmatrix} \cos\h & -\sin\h & 0 \\ \sin\h & \cos \h & 0 \\ 0 & 0 & 1\end{bmatrix} 
=\Theta.   
\]  
So $\Theta$ and $\Theta^{-1}$ are symplectically equivariant SHam-1 actions on 
$\cW_n$. }
\hfill{$\Box$}
\end{rem} 

Consider the standard projection 
\[ 
 \pi':\cW_n\subset \C^3\to \C ,\quad \pi'(z_1,z_2,z_3)=z_3  \quad \text{ restricted to }\cW_n. 
 \] 
Observe that the value of $z_3$ is constant along any integral curve $\ell$ of $\nabla\eh$ and hence on the cylinder $Orb_\Theta(\ell)$ degenerated or  not. So for $z\in \C_{z_3}$ the preimage 
$C_z:=(\pi')^{-1}(z)$ is a $\Theta$-invariant cylinder degenerate only at 
point $p_k=(0,0,\xi^k)$ if $z=\xi^k$ for some $k=0,1,\dots,n$, 
with tangent space 
$T_pC_z=span\{ X,\nabla\eh\}$ when $p\neq p_k$. 
\\

Consider the level set  $\eh^{-1}(0)$. Without loss of generality, 
we may restrict to a $\Theta$-invariant Stein domain 
\[
W_n\subset \cW_n, \quad W_n=(\pi')^{-1}(D)\cap h^{-1}([-c,c]) 
\] 
for some $c>0$, $D=\{ |z|\leq r\}$ for some $r>1$ and $Fix(\Theta)\subset D$. \\ 

On $W_n$ consider the Riemann surface 
\[ 
\Sigma=\{ z_2=0\} \cap W_n\subset\eh^{-1}(0)
\] 
which contains $Fix(\Theta)$ and is transversal to 
$\Theta$-orbits except at $Fix(\Theta)$. 
Let $\Sigma_\h:=\Theta_\h(\Sigma)$, then  
$\Sigma_{\h+\pi}=\Sigma_\h$ for each $\h$, and $\Sigma_{\h}\cap \Sigma_{\h'}=Fix(\Theta)$ if  $\h\neq \h'\mod \pi$. 
Moreover, 
\[ 
\eh^{-1}(0)=\cup_{\h\in\R/\pi\Z}\Sigma_\h. 
\]

For each $\h$  
the symplectic reduction map  $\fr_0|_{\Sigma_\h}:\Sigma_\h\to \C$
  is a $2:1$
 branched covering map branching at $Fix(\Theta)$, 
 with branching locus $\{ \xi\in \C\mid \xi^{n+1}=1\}\subset \C$. 
 Let  $g$ and  $b$ be the genus and boundary components of $\Sigma$ 
(hence each $\Sigma_\h$), then $2-2g-b=2-(n+1)$. 
Observe that the subgroup $K:=\{ \pm Id\}\subset \Theta$ acts on $\pa\Sigma$, $\pa\Sigma/K$ is a simple closed circle which is diffeomorphic to $\pa D$ under the map $\fr_0$.  Therefore $b=1$ or $2$, 
$\Sigma$ is a surface of genus 
$g=\lfloor \frac{n}{2}\rfloor$ with $b=(n+1)-2g$ boundary components.
More explicitly, 
\[ 
(g,b)=\begin{cases} (m, 1) & \text{ if $n=2m\geq 0$ is  even}, \\ 
(m, 2) & \text{ if $n=2m+1\geq  1$ is odd}. 
\end{cases} 
\]

To see the $A_n$-string of Lagrangian spheres in $W_n\subset\cW_n$, 
let 
\begin{equation} \label{stdS}
\gamma_k=\Big\{ e^{i\h}\in \C_{z_3} \Big | \frac{2(k-1)\pi}{n+1}\leq \h\leq \frac{2k\pi}{n+1}\Big\}, \quad k=1,\dots,n, 
\end{equation} 
$\gamma_k$ is an arc between $\xi^{k-1}$ and $\xi^k$. 
Orient $\gamma_k$ so that its tangents point from 
$\xi^{k-1}$ to $\xi^k$. Then 
$\eS_k:=(\fr_0)^{-1}(\gamma_k)$ is a $\Theta$-invariant 
Lagrangian sphere, $\hat{\gamma}_k:=\eS_k\cap 
\Sigma_0$  a simple closed curve,  and  $\eS_k\cap \Sigma_\h=
\Theta_\h(\hat{\gamma}_k)$,  $\eS_k=\cup_{\h\in \R/\pi\Z}\Theta_\h(\hat{\gamma}_k)$. 
The union $\cup_{k=1}^n\eS_k$ 
is an $A_n$-string of Lagrangian spheres. 
$\cW_n$ is  an {\em $A_n$-surface singularity} as in \cite{Wu}. \\

\noindent
{\bf Proof of Theorem  \ref{M=Wn}} \

Let $(M,\omega,\Phi)$ be a connected and 1-connected exact SHam1-space with $n+1$ fixed points, $n\in \N\cup \{ 0\}$, and $X$ the 
Hamiltonian vector field associated to $h$. 
Let $S=S_1\cup S_2\cdots 
\cup S_n$ be an $A_n$-string of Lagrangian spheres in 
$h^{-1}(0)\subset (M,\omega, \Phi)$ associated with $(\Phi,h)$, and  
$\eS:=\eS_1\cup \eS_2\cdots\cup \eS_n$ the $A_n$-string of Lagrangian spheres in $(\cW_n,\omega_n, \Theta)$, say, induced from 
(\ref{stdS}). 
There exists a diffeomorphism $\phi:S\to \eS$ so that 
$\phi(S_i)=\eS_i$ for $i=1,\cdots,n$, and 
 $\Theta\circ\phi=\phi\circ \Phi$ on $S$. Identify $\cW_n
 =\cup_{i=1}^nT^*\eS_i=\cup_{i=1}^nT^*S_i$ 
 with the plumbing of $n$ copies of 
the cotangent bundle $T^*S^2$ of Lagrangian $2$-spheres. 
By Darboux theorem we may extend $\phi$ to a symplectomorphism, also denoted as $\phi$, $\phi:U\to \phi(U)\subset \cW_n$, 
from a neighborhood $U$ of $S_i\cup S_{i+1}$ to a neighborhood $\phi(U)$ of 
$\eS_i\cup\eS_{i+1}$ for each $i$. 
Since a diffeomorphism between two Lagrangian surfaces 
always extends to a symplectomorphism between 
their cotangent bundles, 
we can extend $\phi$ to a symplectomorphism $\phi:M \to \cW_n$, 
sending $T^*S_i$ to $T^*\eS_i$ so that $\phi^*\omega_n=\omega$.  
\hfill{$\Box$} \\


%

\noindent 
{\bf Proof of Theorem  \ref{symp-equiv}\  (i)}\ 
  
Recall that the $n=0$ case follows from Proposition  \ref{G=eG}, 
so we may assume that $n\geq 1$. 
With $\phi$ in the proof of Theorem \ref{M=Wn}  we may identify $(M, \omega, \Phi, h)$ 
with $(\cW_n,\omega_n, \Phi, h)$, 
so that $\Phi$ acts on $(\cW_n,\omega_n=d\alpha_n)$ as an $S^1$-group of Hamiltonian diffeomorphisms $\Phi_\h\in Ham(\cW_n,\omega_n)$, $\Phi_0=id$. \\

Both 
$\Phi$ and $\Theta$ actions preserve $\eS=\eS_1\cup\cdots \cup \eS_n$ with 
$Fix(\Phi)=Fix(\Theta)=\{ p_0,p_1,\cdots,p_n\}$, $p_0\in \eS_1$, $p_i=\eS_i\cap \eS_{i+1}$ for $i=1,\cdots, n-1$, and $p_n\in \eS_n$.  
Since both $\Phi$ and $\Theta$ act on each $\eS_i$ as rotations with 
$p_{i-1}$ and $p_i$ fixed, by modifying $\Phi$ if necessary we may assume that $\Phi_\h=\Theta_\h$ on $\eS$ for all $\h$. The Euclidean metric 
$g_0$ on $\C^3$ restricts to a $\Theta$-invariant, $\omega_n$-compatible 
Riemannian metric on $\cW_n$. Also let $g$ denote a $\omega_n$-compatible metric on $\cW_n$ which is also $\Phi$-invariant. 
We may assume that $g=g_0$ on $T_{\eS_i} \cW_n$ the tangent bundle of $\cW_n$ over $\eS_i$ for all $i$. Let $N(\eS_i)$ denote the normal 
bundle of $\eS_i\subset \cW_n$ determined by the metric $g_0$, which 
can also be identified with the cotangent bundle $T^*\eS_i$. Note that 
\begin{equation} \label{i-j} 
N_p(\eS_i)=T_p\eS_j, \qquad T_p\eS_i=N_p(\eS_j) \ \text{ at  $p=\eS_i\cap \eS_j$ when $|i-j|=1$}. 
\end{equation}

Recall that the action of $\Theta$ on $\cW_n$ is linear. Since $\Phi_\h=\Theta_\h$ 
on $\eS_i$, the linearized action $\Phi_*$ of $\Phi$ acts on $N(S_i)$ in the  same way as $\Theta=\Theta_*$ does. Denote $x_\h:=\Phi_\h(x)=\Theta_\h(x)$ and $v_\h:=(\Phi_\h)_*(v)=\Theta_\h(v)\in N_{x_\h}(\eS_i)$ for  $x\in \eS_i$, $v\in N_x(\eS_i)$.  We have   
for $v\in N_x(\eS_i)$ 
\[ 
(\Phi_\h)_*=\Theta_\h:N_x(\eS_i)\to N_{x_\h}(\eS_i), \quad v\to v_\h=\Theta_\h(v). 
\] 

At $p_i\in Fix(\Phi)$ we have 
\[ 
(\Phi_\h)_*=\Theta_\h:T_{p_i}\cW_n\to T_{p_i}\cW_n 
\] 
and the action preserves the splitting $T_{p_i}\cW_n=T_{p_i}\eS_i\oplus N_{p_i}(\eS_i)$. 
Below we consider exponential maps associated to $\eS\subset \cW_n$.
\\ 
 
%
%
%
For each $i$ 
consider the exponential map over $\eS_i$ relative to $g$, 
so that for $x\in \eS_i$ and $v\in N_x(\eS_i)$, 
\[ 
\fe_i:=\exp_i :\ N(\eS_i)\to \cW_n, \quad v\to \gamma_v(1)
\]  
where $\gamma_v(1)$ is the time $t=1$ image of the geodesic $\gamma_v(t)$  relative to $g$ with $\gamma_v(0)=x$ and $\dot{\gamma}_v(0)=v$.  \\ 

If $n>1$ then at each $p_i\in Fix(\Phi)$ with $1\leq i<n$ we also include  the exponential map 
\[ 
\fe_{p_i}:=\exp_{p_i}:T_{p_i}\cW_n\to \cW_n
\] 
 relative to any $\Phi$-invariant $\omega_n$-compatible Riemannian metric $g$: for $v\in T_{p_i}\cW_n$
\[ 
\fe_{p_i}(v):=\gamma_v(1)
\] 
is the time $t=1$ image of the geodesic $\gamma_v(t)$ relative to 
the metric $g$ with $\gamma_v(0)=p_i$ and  $\dot{\gamma}_v(0)=v$. \\

With (\ref{i-j}) the $\fe_{p_i}$'s and $\fe_i$'s piece together well to form a differentiable map  $\fe_\eS$ from the cotangent/normal  bundle 
\[ 
N(\eS):=\begin{cases} 
N(\eS_1) & \text{if  $n=1$}  \\ 
\cup_{i=1}^n N(\eS_i)\cup_{i=1}^{n-1}T_{p_i}\cW_n & \text{if $n>1$}
\end{cases} 
\] 
 into $\cW_n$. There exists $\epsilon>0$ so that the map 
\[ 
\fe_\eS:=\exp_\eS: \ V:=\{ v\in N(\eS)\mid |v|<\epsilon\} \to \cW_n
\] 
is a diffeomorphism from $V$ onto the open set $U:=\fe_\eS(V)$. 
We can also consider the exponential map $\fe^0_\eS:=\exp^0_\eS$ relative to the 
metric $g_0$, and for $v\in N(\eS)$ based at $x\in \eS$, we write $\gamma^0_v(t)$ as the geodesic relative to $g_0$ with $\gamma^0_v(0)=x$ and $\dot{\gamma}^0_v(0)=v$. \\

Since $\Phi_*=\Theta$  and $g=g_0$ on $T_\eS\cW_n$,  $\gamma_v$ and $\gamma^0_v$ are tangent at the base point $x$ of $v$. For 
$v\in N(\eS)$ denote $v_\h:=(\Phi_\h)_*(v)$, $v_0=v$, then 
\[ 
(\Phi_\h)_*(tv)=tv_\h, \quad 
(\Phi_{\h_1+\h_2})_*(v)=(\Phi_{\h_1})_*(\Phi_{\h_2})_*(v)=(\Phi_{\h_1})_*(v_{\h_2})=v_{\h_1+\h_2}. 
\]
As $g$ is $\Phi$-invariant, $\Phi_\h$ maps $g$-geodesics 
to $g$-geodesics, preserving the arc length. Also, 
$\Phi_\h=\Theta_\h$ when restricted to $\eS$, we have $\Phi_\h(\gamma_v)=
\gamma_{v_\h}$, and 
for $v\in N(\eS)$ based at 
$x\in  \eS$, the map $\fe_\eS$ takes $\gamma^0_v(1)$ to $
\gamma_v(1)$, and $\gamma^0_{v_\h}(1)$ to $\gamma_{v_\h}(1)$. \\ 

The domain $U=\fe_\eS(V)$ is $\Phi$-invariant since $V$ is $\Phi_*=\Theta$-invariant and $g=g^0$ on $V\subset N(\eS)$. Also we may identify 
$(V\subset N(\eS),\omega_n)$ with its image  $(\fe_\eS(V)=U, \omega_n)$.  Then on $U$
\[ 
\Phi=\fe_\eS\circ \Theta\circ \fe_\eS^{-1} . \label{Phi+} 
\]
Let 
\[ \omega':=(\fe_\eS^{-1})^*\omega_n=(\fe_\eS^{-1})^*\sum_{j=1}^3 dx_j\wedge dy_j=\sum_{j=1}^3 dx'_j\wedge dy'_j, 
\]
 where  $x'_j:=x_j\circ \fe_\eS^{-1}$ and $y'_j:=y_j\circ \fe_\eS^{-1}$. $\omega'$ is 
symplectic on $U$, invariant under $\Phi$-action, and the action of 
$\Phi$ on $U$ is linear with respect to $\omega'= \sum_{j=1}^3 dx'_j\wedge dy'_j$. 
Let the pair $(\eX,\eh)$ denote  the Hamiltonian vector field and corresponding Hamiltonian function associated to the 
$\Theta$-action, then 
\[ 
(X:=(\fe_\eS)_*\eX, \ h:=\eh\circ \fe^{-1}_\eS)\] 
 is the the corresponding Hamiltonian 
vector field and Hamiltonian function associated to the $\Phi$-action on $U$. We may extend $\fe_\eS$ to a diffeomorphism $\fe'_\eS:\cW_n\to \cW_n$  with $\Phi$-invariant compact support containing $U$. 
Let 
\[ 
\Phi':=\fe'_{\eS}\circ \Theta\circ {\fe'_{\eS}}^{-1}=p. 
\] 
$\Phi'$ acts as a SHam1--action  $\cW_n$, with $\Phi'=\Phi$ on $U$ and outside the compact support of $\fe'_\eS$. \\

Note that $(\fe_\eS)_*:=d\fe_\eS=id$ when 
restricted to $\eS$, so $\omega'=\omega_n$ on $\eS$. Both $\omega'=d\alpha'$ 
and $\omega_n=d\alpha_n$ are $\Phi$-invariant exact symplectic forms. 
 Up to averaging via the 
$\Phi$-action we may assume that both primitive 1-forms $\alpha_n$ 
and $\alpha'$ are $\Phi$-invariant. 
Consider the smooth family of differential forms $\omega_t:=(1-t)\omega_n+t\omega'$, $0\leq t\leq 1$. Recall that $\omega'=\omega_n$ on $\eS$. Then 
similar to the proof of Proposition \ref{G=eG} (ii), by shrinking $U$ if 
necessarily, we may assume that $\omega_t$ are symplectic on $U$ for all $0\leq t\leq 1$, and  there exists a smooth  isotopy 
$\phi_t:U\to M$, $t\in [0,1]$, with $\phi_0=id$, $\phi_t=id$ on $\eS$ 
for all $t$, and 
\[
\phi_1^*\omega'=\omega_n. 
\] 
%
%
Therefore, any effective SHam1-action on $(\cW_n,\omega_n)$ is 
linear near $\eS$ up to conjugation with a smooth isotopy $\phi_t$ with 
$\phi^{-1}_1\circ \Phi\circ\phi_1=\Theta$ on $U$. In addition $(U,\omega', \Phi)$ and $(V, \omega_n, \Theta)$ are 
symplectically equivariant via the symplectomorphism 
$\fe_\eS:(U,\omega')\to (V,\omega_n)$. 
\hfill{$\Box$} \\

In the following we will show that on $(\cW_n,\omega_n)$ with 
$n=0,1$, any semi-free SHam1-action, which we denote it as $\Phi$, up to conjugations is  linear and, following\ref{symp-equiv}(i), is symplectically equivariant to the $\Theta$-action on any prescribed 
relatively compact domain containing (i) the fixed point of the $\Theta$-action if $n=0$, and (ii) a $\Theta$-invariant Lagrangian sphere if 
$n=1$. \\

\noindent
{\bf Proof of Theorem  \ref{symp-equiv}\ (ii)(iii)} \\ 

\noindent {\bf (ii): $n=0$}.  We denote $q=(x,y)$ for $q\in (\C^2, \omega=\sum_{j=1}^2dx_j\wedge dy_j)$, where 
$x=(x_1,x_2)$, $y=(y_1,y_2)$, also $|x|=\sqrt{x_1^2+x_2^2}$, 
$|y|=\sqrt{y_1^2+y_2^2}$. \\ 

Recall that $\cW_0: \ z_1^2+z_2^2+z_3=1$. 
The embedding  
\[ \iota:\C^2\to \C^3, \quad \iota(z_1,z_2)=(z_1,z_2,1-z_1^2-z_2^2)
\] 
 is a symplectomorphism between $(\C^2,\omega=\sum_{j=1}^2dx_j\wedge dy_j)$ and its image $(\cW_0,\omega_0)$, $\cW_0=\iota(\C^2)$. Conversely the projection map $\pi:\cW_0\to \C^2$, $\pi(z_1,z_2,z_3)=(z_1,z_2)$ is the inverse 
symplectomorphism, $\pi\circ \iota=id_{\C^2}$, $\iota\circ \pi=id_{\cW_0}$,  $\pi^*\omega=\omega_0$. \\ 

Consider The Liouville vector field $Y:=\frac{1}{2}\sum_{j=1}^2 x_j\pa_{x_j}+y_j\pa_{y_j}$ on $\C^2$. 
The differential $\iota_*:T\C^2\to T\cW_0$ induces a Liouville vector 
field $Y_0$ on $\cW_0$ given by $Y_0:=\iota_*Y$. Both $Y$ and $Y_0$ are $\Theta$-invariant. Recall from Proposition \ref{G=eG} 
$\Phi$ and $\Theta$ are symplectically equivariant when restricted to some 
open neighborhood $U$ of the fixed point. Since $\C_2$ and $\cW_0$ 
are symplectomorphic $\Theta$-spaces, and any semi-free SHam1-action on $\cW_0$ descends to one in $\C^2$, 
so it suffices to compare $\Phi$ and $\Theta$ in $(\C^2,\omega)$. 
Then by Proposition \ref{G=eG} (ii) up to conjugation $\Phi$-action is 
 equivariant to $\Theta$-action and hence is linear on some open neighborhood 
$U_0$ of $0\in \C^2$, i.e., there exists a diffeomorphism $\fe_0:U_0\to U$ from a $\Theta$-invariant open neighborhood of $0\in \C^2$ onto a $\Phi$-invariant open neighborhood $U$ of $0\in \C^2$, $\fe_0(0)=0\in \C^2$, and 
\[ 
\Phi_\h=\fe_0\circ \Theta_\h\circ \fe_0^{-1} \quad \text{on $U$}, \quad\forall \h, 
\] 
or equivalently, 
\[ 
\Theta_\h=\fe_0^{-1}\circ \Phi_\h\circ \fe_0 \quad \text{on $U_0$}, \quad \forall \h. 
\]

Let $\rho:[0,\infty)\to [0,1]$ be a $C^\infty$-function so that 
$\rho'(t)\leq 0$ for $t\in [0,\infty)$, and for some 
$0<r_1<r_0<\infty$,  
$\rho(t)=1$ if $t\leq r_1$, $\rho(t)=0$ if $t\geq r_0$. 
Let $\tilde{\rho}:\C^2\to [0,1]$ be the $C^\infty$-function given by 
$\tilde{\rho}(q)=\rho(|q|^2)$. Let $\tilde{Y}:=\tilde{\rho}Y$ and $\psi_t$ be 
the smooth isotopy on $\C^2$ associated to $\tilde{Y}$, $\psi_0=id$, 
$\frac{d\psi_t}{dt}=\tilde{Y}\circ\psi_t$.  For $q=(x,y)\in \C^2$, 
\[ 
\psi_t(q)=(e^{\frac{t}{2}\tilde{\rho}(q)}x, e^{\frac{t}{2}\tilde{\rho}(q)}y). 
\] 
For $r>0$ and $r<r_1$ there exists $t_r<0$  such that 
whenever $|q|^2\leq r$, 
\[ 
\psi_{t_r}(x,y)=(e^{t_r/2}x, e^{t_r/2}y) \in U_0. 
\] 
Let $\Psi_r:=\fe_0\circ\psi_{t_r}$, 
then for $q=(x,y)$ with $|q|^2\leq r$, 
\begin{equation} 
\begin{split} 
(\Psi_r^{-1}\circ \Phi_\h\circ \Psi_r)(x,y) & =  \psi_{t_r}^{-1}\circ \fe_0^{-1}\circ \Phi_\h \circ 
\fe_0\circ \psi_{t_r}(x,y) \\ 
 & =  \psi_{t_r}^{-1}\circ \fe_0^{-1}\circ \Phi_\h \circ \fe_0(e^{t_r/2}x, e^{t_r/2}y) \\ 
  & = \psi_{t_r}^{-1}\circ \Theta_\h (e^{t_r/2}x, e^{t_r/2}y) \\ 
  & = \psi_{t_r}^{-1}(e^{t_r/2}x_\h, e^{t_r/2}y_\h) \\ 
  & = (x_\h,y_\h) = \Theta_\h(x,y) 
\end{split} 
\end{equation} 
where 
\begin{equation} 
\begin{split} 
x_\h & :=  (x_1\cos\h+x_2\sin\h, -x_1\sin\h+x_2\cos\h), \\ 
y_\h & :=  (y_1\cos\h+y_2\sin\h, -y_1\sin\h+y_2\cos\h).  
\end{split} 
\end{equation} 
So 
\[ 
\Psi_r^{-1}\circ \Phi\circ \Psi_r =\Theta  \quad \text{on } \{ (q=(x,y)\in \C^2 | |q|^2\leq r<r_1\}\subset \C^2.   
\] 

By allowing $r_1$ and $r_0$ to be arbitrarily large we 
conclude the that  any semi-free SHam1-action on $(\C^2,\omega)$ with $Fix_\Phi=\{ 0\}$ is 
linear up to conjugation and symplectically equivariant to the standard linear action of $\Theta$ on open 4-ball $B_r(0)=\{ q\in \C^2\mid |q|^2<r\}$ centered at $0$ with any prescribed radius $\sqrt{r}$.  This 
completes the proof of Case 1. 
\\

\noindent {\bf (iii): $n=1$}. \ 
Consider $(\cW_1,\omega_1)$ where $\cW_1: \{ z_1^2+z_2^2+z_3^2=1\}\subset \C^3$ is given by the equations 
\[ 
\sum_{j=1}^3x_j^2=1+\sum_{j=1}^3y_j^2, 
\quad \sum_{j=1}^3x_jy_j=0. 
\] 
For convenience we use the notations  $x:=(x_1,x_2,x_3)$, $y:=(y_1,y_2,y_3)$, $x^2:=\sum_{j=1}^3x_j^2$ and $y^2:=\sum_{j=1}^3y_j^2$. \\ 

$\cW_1$ is indeed the cotangent bundle of the sphere 
$S:=\{ \sum_{j=1}^3x_j^2=1\} \subset\R^3$ which is Lagrangian 
viewed as a submanifold of $(\cW_1,\omega_1)$, $\Theta$ acts on 
$S$ with two fixed points $(0,0,\pm 1)$.  Consider the vector field 
on $\cW_1$: 
\[ 
Y:=\sum_{j=1}^3y_j\pa_{y_j}. 
\] 
The 1-form 
\[ 
\alpha:=\iota_{Y}\omega_1=-\sum_{j=1}^3y_j\pa_{x_j}
\] 
satisfies $d\alpha=\omega_1$, 
is a $\Theta$-invariant primitive 1-form of $\omega_1$, 
so $Y$ is a Liouville vector field on $\cW_1$ which vanishes precisely 
on the Lagrangian sphere $S$. Moreover, $Y$ is invariant under the 
$\Theta$-action, which can be checked by direct computation.

\begin{rem} 
The vector field $Y:=\sum_{j=1}^3y_j\pa_{y_j}$ on $\cW_1$ is indeed 
the gradient vector field $\nabla f$ of the $C^\infty$-function 
$f(x,y)=\sum_{j=1}^3y_j\pa_{y_j}$ with respect to the Riemannian metric which is the pullback of the Euclidean metric on 
 $\C^3$ associated to the inclusion map $\iota:\cW_1\to \C^3$. 
\end{rem} 


It is known that all Lagrangian spheres in the cotangent bundle $T^*S^2$ are hamiltonian isotopic \cite{Hind}, so are those in $\cW_1$. 
Recall from proof of Theorem  \ref{symp-equiv}\ (i) that up to conjugation any semi-free SHam1-action on $\cW_1$ are symplectically equivalent to the standard $\Theta$-action when restricted to some open 
neighborhood  $U$ of the Lagrangian  sphere $S$. We may assume that 
$(x,y)\in U$ whenever $|y|:=\sqrt{y^2}<\delta$ for some $\delta>0$. 
\\ 

Let $\eta_t$ be the time $t$ 
map of the flow of $Y$ with $\eta_0=id_{\cW_1}$. $\eta_t$ is $\Theta$-invariant for all $t$. For $(x,y)\in \cW_1$ 
\[ 
\eta_t(x,y)=(e^sx,e^ty), \quad e^{2s}x^2=1+e^{2t}y^2, \quad s,t\in \R. 
\] 
Consider a 
$C^\infty$-function $\rho:[0,\infty)\to [0,1]$ such that 
$\rho'\leq 0$, and for some $0<r_1<r_0<\infty$, 
$\rho(r)=1$ if $0\leq r\leq r_1$,   $\rho(r)=0$ if $r_0\leq r$. let $\tilde{\rho}:\cW_1\to [0,1]$ be defined as 
$\tilde{\rho}(x,y):=\rho(|y|^2)$. Let $Y_\rho:=\tilde{\rho}Y$. \\ 

Let $\psi_t$ be the time $t$ map of the flow of $Y_\rho$. $\frac{d\psi_t}{dt}=Y_\rho(\psi_t)$. $\phi_0=id_{\cW_1}$. $\frac{d\psi_t}{dt}=0$ when restricted to $S$. $\psi_t$ is $\Theta$-equivariant for all $t$. For $(x,y)\in \cW_1$, 
\[ 
\psi_t(x,y)=(e^sx,e^{\rho(|y|^2)t}y), \quad e^{2s}x^2=1+e^{2\rho(|y|^2)t}y^2.  
\] 
For $r>0$ and $r<r_1$ there exists $t_r<0$  such that 
$\psi_{t_r}(x,y)\in U$ whenever $|y|\leq r$, 
\[ 
\psi_{t_r}(x,y)=(e^{s_r}x, e^{t_r}y) \quad e^{2s_r}x^2=1+e^{2t_r}y^2. 
\] 
\\
As in the proof of Theorem \ref{symp-equiv}(i)  we may assume that 
$\Phi=\Theta$ when restricted to the Lagrangian sphere $S: x^2=1, y=0$ 
in $\cW_1$, and consider the exponential map $\fe_S$ over $S$ relative to a 
$\Phi$-invariant $\omega_1$-compatible Riemannian metric $g$ on $\cW_1$: 
\[ 
\fe_S:=\exp_S:N(S)\to \cW_1, \quad \fe_S(v)=\gamma_v(1). 
\] 
There exists $\epsilon>0$ so 
\[ 
\fe_S:V=\{ v\in N(S)| |v|<\epsilon\}\to \cW_1
\] 
is a diffeomorphism from $V$ onto the open set $U:=\fe_S(V)$. 
Then on $U$ 
\[ 
\Phi_\h=\fe_S\circ \Theta_\h\circ (\fe_S)^{-1}, \quad \forall \h. 
\] 
Let $F_r:=\fe_S\circ \psi_{t_r}$. Similar to the $n=0$ case, we have 
\[ 
F_r^{-1}\circ \Phi_\h\circ F_r=\Theta_\h \quad \text{on $U_r:=\{ (x,y)\in \cW_1 |\  |y|^2\leq r\}$}. 
\] 

By allowing $r_1$ and $r_0$ to be arbitrarily large we 
conclude that any semi-free special Hamiltonian $S^1$-action on $(\cW_1,\omega_1)$ is linearly up to conjugation and symplectically 
 equivariant to the standard linear action of $\Theta$ on $U_r$ for any prescribed $r>0$.  This 
completes the proof of Theorem  \ref{symp-equiv}. 
\hfill{$\Box$} \\

\section{Non-simply connected cases} \label{nsimply}

Now  we consider the case that $(M,\omega)$ is a connected 
exact symplectic 
4-manifold with $c_1(M,\omega)=0$ but not 1-connected. 
Assume that $(M,\omega)$ equips with a 
semi-free special Hamiltonian $S^1$-action which we denote as $\Phi\subset \Ham(M,\omega)$. \\ 

Since the Maslov condition applies only when nonconstant orbits 
of $\Phi$ are homologically trivial,  we assume  that $Fix(\Phi)$ is nonempty and finite to accommodate the Maslov condition . Let 
\[ 
k=n+1:=|Fix(\Phi)|\geq 1
\]
denote  the number of fixed points of the 
$\Phi$-action, $h:M\to \R$ the moment map associated to $\Phi$, 
and $X=X_h$ the Hamiltonian vector field of $h$ defined by 
$\omega(X,\cdot )=-dh$. \\

Fix any $\lambda\in \Omega^1(M)$ with $d\lambda=\omega$. As in the 
1-connected case we may assume that $\lambda$ is $\Phi$-invariant 
and $Fix(\Phi)\subset h^{-1}(0)$, then  $\lambda$ satisfies $\lambda(X)=h$.  
recall  
\[
\fr_c:= h^{-1}(c)/\Phi, \quad \text{ also \ let} \quad \fr^{-1}_{[-c,c]}:= h^{-1}([-c,c])
\]
 associated to the $\Phi$-action on $M$. $h^{-1}(c)/\Phi$ is topologically an oriented surface 
of genus $g$ with $k$ marked points (which are the fixed points  of the 
$\Phi$-action) and $b\geq 1$ connected boundary components at infinity.  
Since $\Phi$ acts freely on $h^{-1}(c)$ for all $c\neq 0$, 
$\fr_c=h^{-1}(c)/\Phi$ is not 1-connected,  hence 
\[ 
g+b\geq 2, \quad g\geq 0, \ b\geq 1. 
\]

\noindent
{\bf Proof of Theorem \ref{pi1}} 

Recall that $\fr_c=h^{-1}(c)/\Phi$ and and $\fr_{c'}=h^{-1}(c')/\Phi$ are homeomorphic topological surfaces for $c,c'\in \R$. 
Let $B_0:=\fr_0$. 
Up to homotopy equivalence we may identify  $M$ with 
the corresponding $\Phi$-invariant subdomain $B=\fr^{-1}_{[-c,c]}B_0$
 for sone $c>0$,  $B$ is homotopic to $M$. \\ 

\noindent
{\bf Part 1: Handle decomposition of $B$ and Homology of $M$}. \ 
For each of the $k=1+n$ fixed points $p_i$, $0\leq i\leq n$, of the 
$\Phi$-action on $M$, denote by $q_i:=\fr_0(p_i)\in B_0$ 
the image of $p_i$ under the map $\fr_0:M_0\to M_0/\Phi=B_0$. In $B_0$ pick $n+1$ open discs  $D_i$ centered at $q_i$, $i=0,1,\dots,n$,  such that the closures of $D_i$'s are pairwise disjoint closed discs in $B_0$. 
 As $B_0$ can be constructed by attaching 
2-dimensional 1-handles to 2-dimensional 0-handles, we may view $D_i$ as the set of  
2-dimensional 0-handles  of $B_0$, and with $q_i$ as the  core point 
of $D_i$. Then $B_0$ can be obtained by attaching the following 
two types of 1-handles to $\cup_iD_i$:

\begin{enumerate} 
\item  {\em A disjoint union of $n$ 1-handles which connect $\cup_{i=1}^nD_i$ to $D_0$}:

Let   $\tau_i\subset B_0\setminus D_0$, $i=1,\dots,n$, be a set 
of pairwise disjoint embedded  arcs with endpoints  
such that $\tau_i$ is disjoint from $q_l$ if $l\neq i$, $q_i=\pa_+\tau_i$ is one endpoint of $\tau_i$, and $\tau_i$ intersects transversally with  
$\pa D_0$ at the other endpoint $\pa_-\tau_i$ of $\tau_i$. Thicken 
each of $\tau_i$ a bit to get $n$ mutually disjoint 2-dimensional 1-handles $H_{\tau_i}$  with 
$\tau_i$ as the core curve of $H_{\tau_i}$, $H_{\tau_i}\cap D_l=\emptyset$ if  $l\neq i$ or $0$, 
and $H_{\tau_i}\cap D_0$  is the end interval of $H_{\tau_i}$ which contains the endpoint $\pa_-\tau_i$ of $\tau_i$.   
WLOG we may assume  that that $D_i\cup H_{\tau_i}$ is 1-connected for $i=1,2,\dots,n$. 

Let 
\[ 
D:=D_0\cup (\cup_{i=1}^n \tilde{H}_{\tau_i}), \quad \tilde{H}_{\tau_i}:=H_{\tau_i}\cup D_i, 
\] 
$D$ is diffeomorphic to a 2-disc. Note that $Fix(\Phi)\subset \fr^{-1}(D)$,  
and $\fr^{-1}_{[-c,c]}(D)$ can be identified with a 1-connected 
Stein domain $W_n\subset \cW_n$. 

\item  {\em A disjoint union  $(2g+b-1)$ 1-handles  with both ends 
attached to $\pa D_0$, missing all of the handles  $H_{\tau_i}$ in (i)}: 

Denote these 1-handles as $H_{C_j}$ for $j=1,2,\dots, 2g+b-1$, where $C_j$ is the core arc of $H_{C_j}$. Denote the boundary of 
$H_{C_j}$ as $\pa H_{C_j}=\pa_+H_{C_j}\cup \pa_-H_{C_j}$, $H_{C_j}$ 
is attached to $\pa D$ along $\pa_{\pm}H_{C_j}$, where $\pa_{\pm}
H_{C_j}$ is 
the disjoint union of two short intervals containing $\pa C_j$. By handle sliding along $\pa D$ (and thinning $\tilde{H}_{\tau_i}$ and $H_{C_j}$  if necessary) we may assume that all $\pa H_{C_j}$  are attached to $\pa D_0$ and not touching any part of $\pa H_{\tau_i}$. 
\end{enumerate} 

Then the union 
\[
D\cup \Big(\sum_{j=1}^{2g+b-1}H_{C_j}\Big)=B_0
\] 
up to a diffeomorphism.  

Recall $\fr_0:h^{-1}(0)\to h^{-1}(0)/\Phi\cong B_0$ the standard projection. Fix a $\Phi$-invariant $\omega$-compatible Riemannian 
metric on $M$ and let $\nabla h$ denote the gradient vector field of 
the moment map $h$ with respect to the metric. 
$M$ deformation retracts to the hypersurface $h^{-1}(0)$ via flows of gradients $\pm \nabla h$ of  $h$, therefore $H_i(M,\Z)\cong H_i(h^{-1}(0), \Z)$ for all $i$.  \\ 

Now $h^{-1}(0)=\fr_0^{-1}(B_0)$. As $\Phi$  acts freely on $\fr_0^{-1}(B_0)\setminus  Fix(\Phi)$, we have the following results. 

\begin{enumerate} 
\item $\fr_0^{-1}(\tau_i)\subset h^{-1}(0)$ is Lagrangian disc centered at 
$p_i$, $1\leq i\leq n$.  If we extend $\tau_i$ to an embedded arc $\tilde{\tau}_i\subset  D$  with $q_i$ and $q_0$ as endpoints, then $\tilde{S}_i:=\fr_0^{-1}(\tilde{\tau}_i)$ is a Lagrangian 2-sphere with 
points $p_i$  and $p_0$  as poles for $i=1,2,\dots n$. These $n$ spheres 
are homologically independent in $\fr^{-1}_0(D)\subset h^{-1}(0)$. Indeed $\fr^{-1}_0(D)$ is homotopic to the bouquet 
$\cup_{i=1}^n \tilde{S}_i$ of $n$ 2-spheres. 

\item $\fr_0^{-1}(C_j)\subset h^{-1}(0)$ is Lagrangian cylinder with 
boundary attached to $\fr^{-1}_0(\pa D)$. If we extend $C_j$ to  a simple closed curve $\tilde{C}_j\subset  D\setminus \{ q_i\mid i=0,\dots n\}$ then we get an embedded Lagrangian torus $\tilde{T}_j \subset h^{-1}(0)=\fr^{-1}_0(B_0)$ for $j=1,2,\dots, 2g+b-1$. These curves 
$\tilde{C}_j$, $j=1,\dots,2g+b-1$, together represent   a basis of the homology group $H_1(h^{-1}(0))=H_1(B,\Z)\cong \Z^{2g+b-1}$. As $\Phi$-orbits 
in $h^{-1}(0)$ contracts to points in $Fix(\Phi)$, 
$H_1(h^{-1}(0),\Z)\cong H_1(B,\Z)$ and is generated by a lifting of 
the curves  $\tilde{C}_j$  in $\tilde{T_j}$, $j=1,\dots, 2g+b-1$. 
The tori $\tilde{T}_j$'s  are 
also independent on $H_2(h^{-1}(0),\Z)$.   Also $\tilde{T_j}$'s and $\tilde{S}_i$'s are pairwise linearly independent in $H_2(h^{-1}(0),\Z)$, 
and forming a basis for $H_2(h^{-1}(0),\Z)\cong \Z^{2g+b-1+n}$. 
Therefore 
 \[
 H_m(M,\Z)=
 \begin{cases}  0 & \text{ for } m\neq 0,1, 2, \\ 
 \Z^{n+2g+b-1}= H_2(\cup_{i=1}^n\tilde{S}_i,\Z)\oplus \Z^{2g+b-1} & \text{ for } m= 2, \\ 
 \Z^{2g+b-1} & \text{ for } m=1, \\ 
\Z & \text{ for } m=0. 
\end{cases} 
\] 
In particular the homology groups of $(M,\omega,\Phi)$ is completely determined by the triple $(k=n+1, g, b)$ with $k\geq 0$, $g\geq 0$, $b\geq 1$,  and $g+b\geq 2$. 

\end{enumerate}

\noindent
{\bf Parr 2: Stein structure on $(M,\omega=d\lambda, \Phi)$} 
 
Below we study the existence of Stein structure on $(M,\omega=d\lambda, \Phi)$. 
Recall the following theorem due to  Eliashberg \cite{E1}  (see also \cite{Gompf1} Theorem 1.3 and \cite{Gompf2} Theorem 2.3) 
about the existence of a 
Stein structure on a 4-manifold via handle attaching criteria, which 
can be stated as the following after Gompf (\cite{Gompf2} Theorem 2.3):

\begin{theo}[Eliashberg.]  \label{twist} 
An oriented 4-manifold admits a Stein structure if and only if it is diffeomorphic to the interior of a handlebody whose handles all have index $\leq$ 2, and for which each 2-handle is
attached along a Legendrian knot (in the standard contact structure on the relevant boundary 3-manifold)
with framing obtained from the contact framing by adding one left twist.
\end{theo} 
 
Here  a 4-dimensional 2-handle $\eH$ is diffeomorphic to a product space modeled on   $D'\times D''_\epsilon \subset \R^2_{x_1,x_2}\times i\R^2_{y_1,y_2}\subset \C^2$, where $D'=\{ x_1^2+x_2^2\leq 1\}$, called  the (Lagrangian) core disc of $\eH$,  and $D''=\{ y_1^2+y_2^2<\epsilon\}$  for some 
$\epsilon >0$. $\eH$ can be viewed as a trivial $D^2$-bundle 
over $D'$. The basis normal vector fields $\{ \pa_{y_1}, \pa_{y_2}\}$ over $D'$, when restricted to $\pa D'$, gives the canonical framing of the 
symplectic normal bundle of $\pa D'\subset \eH$. In particular, we can 
take either $\pa_{y_1}$ or  $\pa_{y_2}$ as the framing of the $D^2$-bundle over 
$\pa D'$. \\

Below we will show that 
the handle decomposition of $B_0=D_0\cup (\cup_{i=1}^n \tilde{H}_{\tau_i})\cup (\cup_jH_{C_j})$ lifts to a corresponding handle decomposition of 
$M$ as a union of handles 
\[ 
M=\eH_0\cup (\cup_i \eH_{V_i})\cup (\cup_j(\eH_{\gamma_j}\cup\eH_{U_j})) 
\] 
where $\eH_0=\fr^{-1}(D_0)$ is  a 0-handle, $\eH_{V_i}$ are 
2-handles associated to $\tilde{H}_{\tau_i}$,  and for each $j$, 
the pair $(\eH_{\gamma_j}, \eH_{U_j})$, where $\eH_{\gamma_j}$ is a 
1-handle and $\eH_{U_j}$ is a 2-handle, is associated to $H_{C_j}$. 
Moreover,  all the 2-handles $\eH_{V_i}$ and $\eH_{U_j}$ of $M$ 
satisfy the contact framing condition, hence $M$  admits the structure 
of a  Stein manifold. \\

{\bf Case 1: $\eH_{V_i}$.} \ 
Denote by  $V_i:=\fr^{-1}_0(\tau_i)$ the Lagrangian disc in $h^{-1}(0)$ 
with the boundary circle $\pa V_i$ attached to $\fr^{-1}_0(\pa D_0)$. $V_i$ is the core disc of a 4-dimensional 
handle $\eH_{V_i}$ which can be identified with a subdomain of the cotangent bundle over $V_i$. We may assume that these handles are mutually disjoint when attached to $\fr^{-1}(\pa D_0)$.  \\

 Pick any $V_i$ and denote it as $V$, 
also denote  $\tau:=\tau_i=\fr(V)$, $p:=p_i$.  
We may parametrize 
$V$  as 
\[ 
V=\{ (x_1,x_2)\mid x_1^2+x_2^2\leq 1\}, \quad p=(0,0)\in Fix(\Phi)
\] 
and identify 
\[ 
X|_V=x_1\pa_{x_2}-x_2\pa_{x_1}. 
\] 
Let  $(y_1,y_2)$ be the fiber coordinates of the cotangent 
bundle $T^*V$ dual to $(x_1,x_2)$. Then $V$ can 
be identified with the core disc of the 4-dimensional 2-handle 
\[ 
\eH_V:=\fr^{-1}_0(H_\tau)\cong V\times D^2_y\subset T^*L, \quad 
D^2_y=\{ (y_1,y_2)\mid y_1^2+y_2^2< \epsilon^2\} 
\] 
for some $0<\epsilon$, 
with 
$\pa_-\eH_V:=\pa V\times D^2_\epsilon$ attached to $\fr^{-1}(\pa D_0)\subset \pa W_0$, where $W_0=\fr^{-1}(D_0)$. 

Identify $\eH_V$ as a subdomain of $\C^2$. Then on $\eH_V$ we may 
take $\omega=d\lambda$, with  $\lambda=\frac{1}{2}\sum_{i=1}^2(x_idy_i-y_idx_i)$,  identify the $\Phi$-action with the standard $\Theta$-action on $\C^2$, and take the Euclidean metric as the $\Phi$-invariant 
$\omega$-compatible Riemannian metric on $\eH_V$. Then 
\[ 
\nabla h|_{\pa V}=-x_1\pa_{y_2}+x_2\pa_{y_1}. 
\]

Note that both $X$ 
and $\nabla h$ are tangent to $\fr^{-1}(\pa D_0)$ and hence to 
$\pa_-\eH_V$. 
The 1-form $\lambda$ restricts to a contact 1-form $\alpha$ on $\pa V\times D^2_\epsilon$. $\lambda(X)=0$ along $\pa V$
so $\pa V$ is a Legendrian curve to the contact form.  
As the outward normal vector field 
to $\fr^{-1}(D_0)=W_0$ near  $\pa V$ is $-x_1\pa_{x_1}-x_2\pa_{x_2}\subset \ker\lambda |_{\pa V}$, the contact plane field along $\pa V$ is 
spanned by the symplectic pair $\{ \nabla h, X\}$.  So the contact 
framing along $\pa V$ (oriented by $X$) is given by the vector field 
$\nabla h|_{\pa V}=-x_1\pa_{y_2}+x_2\pa_{y_1}$. \\

Identify $V$ as the 0-section of the cotangent bundle $T^*V\supset \eH_V$. 
The normal bundle $N_V$ of $V\subset T^*V$ is trivial with 
$\{ \pa_{y_1}, \pa_{y_2}\}$ as the basis fields of $N_V$. 
$N_V$ restricted to $\pa V$  is the symplectic normal bundle 
$SN^*(\pa V)=\pa_{y_1}\wedge \pa_{y_2}|_{\pa V}$ of $\pa V$ with the natural trivialization/framing 
given by $\pa_{y_1}$. \\ 

$SN^*(\pa V)$ is then identified with the normal bundle 
$N_{\pa V/\pa W_0}$ of $\pa V\subset \pa W_0$ 
upon the attaching of $\pa_-\eH_V=\pa V\times D^2_y$ to $\pa W_0$ 
along $\pa V$. 
As we go once along $\pa V$ the 
contact framing $\nabla h$ makes one positive full rotation relative to the 
canonical framing $\pa_{y_1}$. In other words, the canonical framing 
$\pa_{y_1}$  along $\pa V$ is obtained from the contact framing by adding one left twist. So the natural framing along 
$\pa V$ given by $\pa_{y_1}$ is $-1$ 
relative to the contact framing along $\pa V$, which meets with 
the Stein condition on the attachment of  4-dimensional 2-handles. \\ 

The above result on the framing criteria along $\pa V\subset \pa_-\eH$ apply to $V_i$ for all $i=1,2,\dots,n$ as well. Each of the 2-handlebodies $\eH_{V_i}$ has its core disc  
$V_i$ 
attached along a Legendrian knot (in the standard contact structure on the relevant boundary 3-manifold)
with framing obtained from the contact framing by adding one left twist.
In particular, $\eH_0\cup (\cup_i \eH_{V_i})$ admits a Stein structure by Theorem \ref{twist}. \hfill{$\Box$} \\

{\bf Case 2: $\eH_{C_j}$.} \ 
Pick any one of the 2-dimensional handles $H_{C_j}$ and denote it 
as $H_C$, where $C$ is the core curve of $H_C$. 
We will see that $\fr^{-1}(H_C)$ can be identified with the union of a 4-dimensional 1-handle $\eH_\gamma$  with 
core curve $\gamma \subset h^{-1}(0)$ and a 4-dimensional 
2-handle $\eH_U$ with core disc $U\subset h^{-1}(0)$, 
 and $\pa U$ is  attached to $\pa \eH_\gamma\cap h^{-1}(0)$. \\

The preimage $L:=\fr^{-1}_0(C)\subset h^{-1}(0)$ is a 
Lagrangian annulus $L\cong I\times S^1$  foliated by $G$-orbits 
generated by the Hamiltonian vector field $X$. We may parametrize 
$L$  as 
\[ 
L=\{ (x_1,x_2)\in \R\times \R/2\pi\Z \mid -1\leq x_1\leq 1, x_2\in \R\} 
\] 
so that $X=\pa_{x_2}$, and each integral curve of $\pa_{x_1}$ is 
a lifting of $C$ in $L$. The boundary of $L$, which is the pair of 
$\Phi$-orbits given by the equations $x_1=\pm 1$ is attached to $\pa W_n$. Let 
\[ 
\gamma:=\{ (x_1,x_2)\mid -1\leq x_1\leq 1, \ x_2=0 \}\subset L. 
\] 
$\gamma$ is a lifting of $C$ in $h^{-1}(0)$. \\

Let  $(y_1,y_2)$ be the fiber coordinates of the cotangent 
bundle $T^*L$ dual to $(x_1,x_2)$. Then $\gamma$ can 
be identified with the core curve of the 4-dimensional 1-handle 
(a product of four intervals) 
\[ 
\eH_\gamma:=\fr^{-1}(H_C)\cong \gamma\times I^3\subset T^*L, \quad 
I^3= I_{|y_1|<\epsilon}\times I_{|x_2|<\epsilon}\times I_{y_2}
\] 
for some $0<\epsilon\ll 1$, 
with 
$\pa_-\eH_\gamma:=\pa\gamma\times I^3$ attached to $\pa M_0$. 
Here 
$y_2$ parametrizes integral curves of $-\nabla h$, 
the negative  gradient vector field of 
$h$ with respect to some $G$-invariant $\omega$-compatible Riemannian metric on $M$. 
Let 
\[ 
I_x:=\gamma\times I_{|x_2|<\epsilon}=\{ (x_1,x_2)\mid -1\leq x_1\leq 1,\  -\epsilon <x_2<\epsilon\}\subset L.  
\] 
Up to a smoothing of the corners, the complement 
\[ 
U:=L\setminus I_x  = [-1, 1]\times [\epsilon , 2\pi-\epsilon]\subset L
\] 
is a Lagrangian disc 
in $h^{-1}(0)$ with 
boundary attached to $\pa W_n\cap h^{-1}(0)$. 
Observe that $U$ is the core 2-disc of the handle 
\[ 
\eH_U:=U\times I_{|y_1|<\epsilon}\times I_{y_2}, 
\] 
with $\pa_-\eH_U=\pa U \times I_{|y_1|<\epsilon}\times I_{y_2}$ attached to $\pa_+\eH_\gamma=\gamma\times \pa I^3$ so that 
 \[ 
 \eH_\gamma\cup \eH_U=T^*L\cap \{ |y_1|<\epsilon\}. 
\]

Orient  
$\pa U=\gamma_-\cup \sigma_+\cup \gamma_+\cup \sigma^-$  
counterclockwise, where 
\begin{alignat}{2}
\gamma_-  & =  \{ -1\leq  x_1\leq 1, \ x_2=\epsilon\}, & \quad  \dot{\gamma}_- & =\pa_{x_1}, \\ 
\sigma_+  & =  \{ x_1=1, \ \epsilon \leq x_2\leq 2\pi-\epsilon\}, & 
\quad  \dot{\sigma}_+ & =\pa_{x_2}=X, \\ 
\gamma_+   &= \{ -1\leq  x_1\leq 1, \ x_2=2\pi-\epsilon\}, & \quad  
\dot{\gamma}_+ & =-\pa_{x_1},   \\ 
\sigma_-  & =   \{ x_1=-1, \ -\epsilon \leq x_2\leq 2\pi-\epsilon\}, & 
\quad  
\dot{\sigma}_- & =-\pa_{x_2}=-X. 
\end{alignat} 

Recall  the $\Phi$-invariant primitive 1-form $\lambda$ of $\omega$. 
Observe that $\fr^{-1}(\pa D)\subset \pa W_n$ is a subdomain of 
$\pa W_n$, 
$\lambda$ restricted to $\pa W_n$ is a contact 1-form  near  $\pa L$, 
 whose contact structure along $\pa L$ is  spanned by the symplectic pair $\nabla h$ and $X$. 
 So along $\sigma_+$ the contact structure is $X\wedge (-\nabla h)=\pa_{x_2}\wedge \pa_{y_2}$, and along $\sigma_-$ the contact structure is $-X\wedge \nabla h=-\pa_{x_2}\wedge( -\pa_{y_2})$. 
 Along the boundary $(\pa_+\eH_\gamma)\cap L=\gamma_-\cup \gamma_+$ 
 the vector field $X$ is normal to $\pa_+\eH_\gamma$,  hence along $\gamma_-$ the contact structure  on $\pa_+\eH_\gamma$ is $\pa_{x_1}\wedge \pa_{y_1}$, spanned by the symplectic pair 
 $\dot{\gamma}_-=\pa_{x_1}$ and $\pa_{y_1}$.  Similarly along $\gamma_+$ the contact structure  on $\pa_+\eH_\gamma$ is $(-\pa_{x_1})\wedge (-\pa_{y_1})$, 
 spanned by 
 the symplectic pair $\dot{\gamma}_+=-\pa_{x_1}$ and $-\pa_{y_1}$. \\

 By smoothing the corners of $U\subset L$ we may identify $U$ 
 with a unit disc $D_U=\{ x_1^2+x_2^2\leq 1\}$ with boundary 
 $\pa D_U$ attached to $\pa(\eH_\gamma\cup W_n)$, so that along the 
 legendrian circle $\ell:=\pa D_U= S^1=\R/2\pi\Z$,  
 \[ 
 \begin{split} 
 \dot{\ell}(\h) & =\cos\h\pa_{x_2}-\sin\h \pa_{x_1},\quad \ell(0)=(1,0) \\ 
  & =x_1\pa_{x_2}-x_2\pa_{x_1}, 
  \end{split}  
 \] 
and the  vector field 
\[ 
\cos\h \pa_{y_2} -\sin\h\pa_{y_1} =x_1\pa_{y_2}-x_2\pa_{y_1} 
\] 
is the contact framing along $\ell$. \\ 

Similar to Case 1, the framing of the symplectic normal bundle $SN^*(\pa U)$ (which is spanned by the 
basis vector fields $\{ \pa_{y_1}, \pa_{y_2}\}$) is given by $\pa_{y_1}$. As we go once along $\ell$ the 
contact framing makes one positive full rotation relative to  $\pa_{y_1}$. So the canonical framing $\pa_{y_1}$ is $-1$ 
relative to the contact framing along $\ell$.  This result applies to all 
$\eH_{C_i}$. 
Combining the  result from Case 1 we conclude that   $M$ admits 
a Stein structure following Theorem \ref{twist}. 
This completes the proof of Theorem \ref{pi1}. \hfill{$\Box$}

\end{document}